%% file: christoffel-regression.tex
\documentclass[10pt]{amsart}

\pdfoutput=1

\usepackage{amsmath,amssymb,graphicx,bbm}
\usepackage{amsthm,verbatim,xcolor}
\usepackage{mathrsfs}

\usepackage[norelsize,lined,boxed,linesnumbered,commentsnumbered]{algorithm2e}

\usepackage[footnotesize,bf]{caption}
\usepackage[left=1.25in,right=1.25in,top=1in]{geometry}

\usepackage{acncommands}

\usepackage{array,multirow,booktabs} 
\theoremstyle{plain}
\newtheorem*{mythm}{Theorem}

\newcommand{\vertiii}[1]{{\left\vert\kern-0.25ex\left\vert\kern-0.25ex\left\vert #1 
    \right\vert\kern-0.25ex\right\vert\kern-0.25ex\right\vert}}

\newcommand{\revision}[1]{{\color{red} \bf #1}}
\renewcommand{\revision}[1]{#1}

\newcommand{\rev}[1]{{\color{red} {#1}}}
\renewcommand{\rev}[1]{#1}

\title[Christoffel Least Squares]{A Christoffel function weighted least squares algorithm for collocation approximations}
\author{Akil Narayan}
\thanks{Akil Narayan. Mathematics Department and Scientific Computing and Imaging Institute, University of Utah, University of Utah, Salt Lake City, UT 84112. A.~Narayan was partially supported by AFOSR FA9550-15-1-0467 and DARPA N660011524053}
\author{John D. Jakeman}
\thanks{John D. Jakeman. Computer Science Research Institute, Sandia National Laboratories, 1450 Innovation Parkway, SE, Albuquerque, NM 87123}
\author{Tao Zhou}
\thanks{Tao Zhou. Institute of Computational Mathematics and the Chinese Academy of Sciences, Beijing, China. T.~Zhou work was supported the National Natural Science Foundation of China (Award Nos. 91130003 and 11201461).}

\begin{document}

\begin{abstract}
  We propose, theoretically investigate, and numerically validate an algorithm for the Monte Carlo solution of least-squares polynomial approximation problems in a collocation framework. Our investigation is motivated by applications in the collocation approximation of parametric functions, which frequently entails construction of surrogates via orthogonal polynomials. A standard Monte Carlo approach would draw samples according to the density \rev{defining the orthogonal polynomial family}. Our proposed algorithm instead samples with respect to the (weighted) pluripotential equilibrium measure of the domain, and subsequently solves a weighted least-squares problem, with weights given by evaluations of the Christoffel function. We present theoretical analysis to motivate the algorithm, and numerical results that show our method is superior to standard Monte Carlo methods in many situations of interest.
\end{abstract}

\maketitle

\input{content/introduction}

\input{content/setup}
\input{content/cls}
\input{content/background}
\input{content/proofs}
\input{content/results}
\input{content/conclusion}

\section{Acknowledgments}
The authors express deep thanks to Dr. Norman Levenberg who provided much insight into weighted pluripotential theory.

This work was supported by the U.S. Department of Energy, Office of Science, Office of Advanced Scientific Computing Research, Applied Mathematics program. Sandia National Laboratories is a multi-program laboratory managed and operated by Sandia Corporation, a wholly owned subsidiary of Lockheed Martin Corporation, for the U.S. Department of Energy National Nuclear Security Administration under contract DE-AC04-94AL85000. 


\bibliographystyle{abbrv}
\bibliography{christoffel-regression}

\end{document}

%% file: content/introduction.tex
\section{Introduction and main results}\label{sec:introduction}

We consider the polynomial approximation of a function $f: \R^d \rightarrow \R$ using a least-squares collocation method. We are particularly interested in the case when the argument to $f$ is a finite-dimensional random variable $z$ (denoted lowercase throughout) with associated probability density function $w$. In this case, approximation of $f(z)$ is typically carried out in a $w$-weighted norm and can be constructed using a Monte Carlo procedure. This problem is particularly germane for parametric uncertainty quantification where $f$ is usually a parameterized function with random parameter $z$ \cite{xiu_numerical_2010}. Constructing a polynomial surrogate is a standard approach and is frequently explored via generalized Polynomial Chaos where $f$ is expanded in a basis whose polynomial elements are orthogonal under the weight $w$ \cite{wiener_homogeneous_1938,xiu_wiener--askey_2002}. Using a collocation procedure to construct this polynomial is advantageous in practical large-scale simulations \cite{xiu_high-order_2005,narayan_stochastic_2015}.

While interpolatory approaches \cite{narayan_stochastic_2012,barthelmann_high_2000} and compressive sampling or $\ell^1$ regularization techniques \cite{doostan_non-adapted_2011,rauhut_sparse_2012} are effective, the least-squares $\ell^2$ regularization procedure is one of the simplest strategies that offers an attractive balance between cost and accuracy.  Many existing methods for least-squares regression in this context concentrate on Monte Carlo approaches where the random variable ensemble $\{z_i\}_i$ is sampled iid according to the weight function $w$ \cite{migliorati_analysis_2014,tang_discrete_2014,chkifa_discrete_2013,migliorati_approximation_2013}. Alternative methods include the use of deterministic point constructions \cite{zhou_multivariate_2014} or strategies involving subsampling from a ``good" high-cardinality mesh \cite{zhou_subsampling_2014}.

This paper presents analysis and computational results for a type of weighted Monte Carlo approach for least-squares polynomial approximation that we call Christoffel Least Squares (CLS). The CLS prescription has two simple ingredients: given a probability weight/density $w$, we sample iid with respect to the (weighted) pluripotential equilibrium measure (\textit{not} iid from $w$), and the weights are evaluations of the Christoffel function from the $w$-orthogonal polynomial family. The concrete procedures are shown in Algorithms \ref{alg:cls-bounded} and \ref{alg:cls-unbounded}. If one writes the least-squares problem in matrix formulation as an algebraic problem, weighting by the Christoffel function is equivalent to normalizing the system matrix so that each row has the same discrete $\ell^2$ norm. The CLS algorithm is applicable for bounded and unbounded domains, with tensor-product or more general non-tensor-product weights and domains.

Our analysis for the CLS method for polynomial approximation is based on the general least-squares theory presented in \cite{cohen_stability_2013}. Given an $N$-dimensional subspace $P$ of $L^2_w(D)$ for some closed set $D \subset \R^d$, let $\phi_n(z)$ denote any orthonormal family for $P$. We let $K(z)$ denote the ``diagonal" of the reproducing kernel of $P$ in $L^2_w$:
\begin{align*}
  K(z) = \sum_{n=1}^N \phi_n^2(z).
\end{align*}
Note that, fixing $D$ and $w$, $K$ does not depend on which orthonormal basis for $P$ is chosen.{\footnote{\rev{With $\boldsymbol{\phi}$ a vector containing the $\phi_n$, then $K(z) = \boldsymbol{\phi}^T \boldsymbol{\phi}$. Thus, any change-of-basis $\boldsymbol{\psi} \gets \boldsymbol{U} \boldsymbol{\phi}$ via any orthogonal matrix $\V{U}$ preserves $K(z) = \boldsymbol{\psi}^T \boldsymbol{\psi} = \boldsymbol{\phi}^T \boldsymbol{\phi}$.}} The analysis in \cite{cohen_stability_2013} shows that a Monte Carlo least-squares approximation method with samples chosen iid from $w$ is stable and accurate with high probability if the number of samples $S$ satisfies
\begin{align}\label{eq:ls-stability-factor}
  \frac{S}{N \log S} \gtrsim C \frac{\|K\|_\infty}{N}
\end{align}
where $C$ is a universal constant and $\|K\|_\infty \triangleq \max_{z \in D} K(z)$. Since $K$ is a reproducing kernel diagonal and $w$ is a probability density, then a lower bound on the value for $\|K\|_\infty$ is $N$:
\rev{ {
\begin{align*}
  N = \int_{D} \sum_{n=1}^N \phi_n^2(z) w(z) \dx{z} \leq \left[ \max_{z \in D} \sum_{n=1}^N \phi_n^2(z) \right] \int_D w(z) \dx{z} = \|K\|_\infty
\end{align*}  
}
}
I.e., the best (smallest) possible value of $\|K\|_\infty/N$ in \eqref{eq:ls-stability-factor} is unity. However, with $P$ a total-degree polynomial space, for many weights $w$ of interest the actual value of this quantity is very large and is quite sensitive to the choice of $w$ (see Figure \ref{fig:maxK-plots}), and therefore makes the requirement for stability computationally onerous.

The CLS algorithm we present in this paper mitigates this situation by leveraging the fact that, for polynomials, the asymptotic behavior of the total-degree space reproducing kernel diagonal is known in great generality. (In this paper, ``asymptotic" means with respect to the polynomial degree.) Let $P_k$ denote the space of polynomials of degree $k$ or less over $D \subset \R^d$, so that $N = \dim P_k = \left(\begin{array}{c} d+k \\ d \end{array}\right)$. We let $K_k$ denote the $L^2_w$ reproducing kernel diagonal of $P_k$. The quantity $N/K_k$ is the (normalized) Christoffel function from the theory of orthogonal polynomials (e.g., \cite{nevai_geza_1986}), and is the eponymn of the CLS algorithm. If $D$ is compact with non-vanishing interior and positive $d$-dimensional Lebesgue measure, and $w$ is continuous on the interior of $D$ and admits an orthogonal polynomial family, then 
\begin{align}\label{eq:K-approximate-asymptotics}
  \lim_{k \rightarrow \infty} \frac{N}{K_k(z)} = \frac{w(z)}{v(z)},
\end{align}
almost everywhere in $D$, where $v(z)$ is the Lebesgue weight function (a probability density) of the pluripotential equilibrium measure of $D$ \cite{berman_bergman_2009-1}. For example, in $d=1$ dimension on the interval $D = [-1,1]$, $v(z)$ is the arcsine or ``Chebyshev" density. The utility of this statement for least-squares approximations is that the non-polynomial functions
\begin{align*}
  \psi_n = \sqrt{\frac{N}{K_k(z)}} \phi_n(z),
\end{align*}
form a basis for approximation in the space $\frac{1}{\sqrt{K_k}} P_k$, and are orthogonal in an $L^2$ space with the modified weight function $w \frac{K_k}{N}$. Owing to \eqref{eq:K-approximate-asymptotics}, the $\psi_n$ are therefore approximately orthonormal with respect to $v$, and so an ``approximate" reproducing kernel diagonal is given by 
\begin{align*}
  \widetilde{K}_k = \sum_n \psi_n^2 = \frac{N\, K_k}{K_k} = N,
\end{align*}
and this therefore attains the the optimal (smallest) supremum value of $N$. Therefore, if we instead perform a Monte Carlo approximation with the $\psi_n$, sampling from $v$, then it may be possible to obtain the optimal sample-count stability criterion from \eqref{eq:ls-stability-factor} for most weights $w$ of interest. This, in a nutshell, is the CLS algorithm. Although we have framed this discussion for compact domains and total-degree polynomial spaces, the CLS method may be applied for general polynomial subspaces on conic unbounded domains with exponential weights.


We present theoretical analysis following the results in \cite{cohen_stability_2013} that crystallizes the motivation above, and accompanying numerical simulations show that the CLS algorithm significantly outperforms standard MC methods in many (but not all) scenarios of practical interest. \rev{For a general polynomial subspace $P$ and its associated reproducing kernel diagonal $K$,} the CLS algorithm performs approximation on the $L^2$ space weighted with $\widetilde{w} = \frac{N}{K(z)} v(z)$. Thus, the theory depends on a measure of discrepancy between $\widetilde{w}$ and $w$. One such measure on the space $P$ \rev{is independent of the function being approximated:} it is the $\widetilde{w}$-Gramian of the \rev{$w$-orthonormal basis} $\phi_n$, an $N \times N$ matrix $\V{R}$ with entries:
\begin{align*}
  (R)_{m,n} = \int_D \phi_m(z) \phi_n(z) \, \widetilde{w}(z) \dx{z}
\end{align*}
Let $\Pi $ denote the $L^2_w$-orthogonal projector onto $P$, where $L^2_w$ is the $w$-weighted $L^2$ space on $D$ with norm $\|f\|^2 = \int_D f^2(z) w(z) \dx{z}$.  Similarly, let $\widetilde{\Pi}$ denote the $L^2_{\widetilde{w}}$-orthogonal projector onto $P$. \rev{A second $w$ versus $\widetilde{w}$ discrepancy measure is data-dependent, the error in projections:
\begin{align*}
  d(f) = \left\|\widetilde{\Pi} f - \Pi f \right\|_w
\end{align*}
which vanishes for any $f \in P$, or for any $f$ when $\widetilde{w} = w$.}

The following is one of our major theoretical results, summarizing our Theorem \ref{thm:bounded-convergence}, and frames accuracy in terms of $d(f)$ and spectral quantities of $\V{R}$.
\begin{mythm}
  Let $D$ be compact. The CLS algorithm, i.e., discrete least-squares approximation by sampling iid from the equilibrium measure $v$ and weighting with the inverse kernel diagonal $N/K(z)$, is stable with high probability if, for any $r > 0$, the number of samples $S$ satisfies 
  \begin{align*}
    \frac{S}{N \log S} \geq C \frac{1+r}{\lambda_{\mathrm{min}}(\V{R})}
  \end{align*}
  where $C$ is an absolute constant. Let $f$ be a function satisfying $|f|\leq L$, and let $\widetilde{\Pi}^S f$ denote the $S$-sample CLS estimator of $f$ on $P$. Then under the sampling criterion above,
\rev{
  \begin{align*}
    \E \left[ \left\|f - T_L(\widetilde{\Pi}^S f) \right\|_w^2\right] \leq \left\|f - \Pi f\right\|_w^2 + \frac{\varepsilon(S)}{\lambda_{\mathrm{min}}(\V{R})} \left\|f - \Pi f\right\|^2_{\widetilde{w}} + \frac{8 L^2}{S^r} + 4\kappa^2(\V{R}) d^2(f) 
  \end{align*}
}
  where $T_L(x)= \mathrm{sgn} (x) \min\left\{|x|, L\right\}$ is a truncation function, and $\varepsilon(S) \sim \frac{1}{r \log S} \rightarrow 0$ as $S \rightarrow \infty$.
\end{mythm}
Above, $\lambda_{\mathrm{min}}(\V{R})$ is the minimum eigenvalue and $\kappa(\V{R})$ the 2-norm condition number of $\V{R}$. Our numerical results for one dimension indicate that both of these quantities are very well-behaved for general $w$ (see Figure \ref{fig:qn-mineig}) and thus are a significant improvement over the standard MC approach criterion \eqref{eq:ls-stability-factor}. \rev{The scalar $r$ introduced in the theorem above is a tunable factor that quantifies the oversampling rate, and is not a novel by-product of the CLS procedure or its theory. }The projection discrepancy term $d(f)$ in the conclusion of the theorem above does not vanish as $S \rightarrow \infty$, \rev{and thus this theory appears uncompetitive with established theory for a ``standard" Monte Carlo (MC) method. (We review some of this theory in Section \ref{sec:background-ls}.) However, our empirical investigation in Section \ref{sec:cls-accuracy} indicates that the error bound above for the CLS procedure has the same magnitude as the MC bounds. And although we cannot yet rigorously show the comparability of the CLS versus MC theory, our numerical results in Section \ref{sec:results} indicate that the CLS algorithm is frequently superior to a standard least-squares MC approach for polynomial approximation.}

Although we have described only the bounded-domain case above, the CLS algorithm is also applicable on unbounded domains with exponential weights. The unbounded case presents no great difficulty in terms of analytical results comparable to the bounded case (see Theorems \ref{thm:unbounded-stability} and \ref{thm:unbounded-convergence}), but the implementation is less straightforward because an explicit formula for the sampling measure (the weighted pluripotential equilibrium measure) is not yet known for weights of interest. Nevertheless, we conjecture the forms of these weights and our simulations yield results that support our conjectures, see Table \ref{tab:cls-sampling} and Section \ref{sec:results}.

%% file: content/setup.tex
\section{Setup}\label{sec:setup}
Let $D \subset \R^d$ be the domain and let $w: D \rightarrow \R$ be a weight function. We assume that the pair $(D,w)$ is ``admissible", by which we mean it falls into one of the following categories:
\begin{itemize}
  \item (\textit{bounded}) $D$ is a compact set of nonzero $d$-dimensional Lebesgue measure and nonempty interior, and $w$ is a continuous function on the interior of $D$ such that $0 < \int_D p^2(z) w(z) \dx{z} < \infty$ for any nontrivial algebraic polynomial $p$.
  \item (\textit{unbounded}) $D$ is an origin-centered unbounded conic domain (i.e., if $z \in D$, then $c z \in D$ for all $c \geq 0$) with nonzero $d$-dimensional Lebesgue measure, and $w = \exp(-2 Q(z))$, with $Q$ satisfying (i) $\lim_{|z|\rightarrow\infty} Q(z)/|z| > 0$, and (ii) there is a constant $t \geq 1$ such that 
    {\begin{align}\label{eq:q-homogeneous}
        Q(c z) &= c^t Q(z), & \forall\;\; z \in D,\; c > 0.
    \end{align}}
\end{itemize}
The condition \eqref{eq:q-homogeneous} states that $Q = -\frac{1}{2} \log w$ is a $t$-homogeneous function. The unbounded case with the homogeneity condition on $Q$ includes the following general family of weights:
\begin{align*}
  Q(z) &= \left\| z\right\|_p^t, & \|z\|_p \triangleq \left( \sum_{k=1}^d |z_k|^p \right)^{1/p},
\end{align*}
for any $p \geq 1$, which includes as special cases the one-sided exponential weight $w(z) = \exp\left(-\sum_j z_j\right)$ on $D = [0,\infty)^d$, and the Gaussian density function $w(z) = \exp\left(-\sum_j z_j^2\right)$ on $D = \R^d$. We will frequently write $|z|$ to mean $\|z\|_2$. 

When citing results from pluripotential theory we will identify the function $\exp(-Q)$ as a ``pluripotential weight function", which will be used a theoretical tool. For the CLS-unbounded case, note that specification of $Q$ uniques defines $w$. Whereas, for the CLS-bounded case we will take $Q=0$ (for any admissible $w$). Such correspondences between $w$ and $Q$ are made in this paper to be consistent with notational conventions in pluripotential theory. 

\subsection{Orthogonal polynomials}
If $(D,w)$ is admissible, then an $L^2_w(D)$ orthogonal polynomial family exists (see., e.g., \cite{ernst_convergence_2012,lubinsky_survey_2007}). For a multi-index $\alpha \in \N_0^d$ we let $\phi_\alpha$ denote the family of polynomials \textit{orthonormal} under the $w$-weighted $L^2$ norm on $D$: 
\begin{align*}
  \left\langle \phi_\alpha, \phi_\beta \right\rangle_w &= \int_D \phi_\alpha \phi_\beta w\, \dx{z} = \delta_{\alpha,\beta},
\end{align*}
with $\|\cdot\|_w$ the corresponding induced norm on $L^2_w$. We implicitly assume that $\deg \phi_\alpha = |\alpha| \triangleq \alpha_1 + \cdots + \alpha_d$, with $\alpha_j$ the components of $\alpha$. We use $\Lambda \subset \N_0^d$ to denote a general multi-index set, with $P \triangleq \mathrm{span} \left\{ \phi_\alpha\; |\; \alpha \in \Lambda\right\}$ \rev{and the associated $L^2_w$-orthogonal projection onto $P$:
  \begin{align*}
    \Pi f &\triangleq \sum_{\alpha \in \Lambda} c_\alpha \phi_\alpha, & c_\alpha &= \int_D f(z) \phi_\alpha(z) \dx{z}.
  \end{align*}
  For notational simplicity, we suppress the explicit dependence of $P$ and $\Pi$ on the index set $\Lambda$.
}

Given any $N$-dimensional subspace $P$ of $L^2_w$, in what follows we will study the ``diagonal" of its reproducing kernel in $L^2_w$. This quantity is given by
\begin{align}\label{eq:K-def}
  K(z) = \sum_{\alpha \in \Lambda} \phi_\alpha^2(z) = \sum_{n=1}^N \phi_n^2(z),
\end{align}
and is not dependent on the choice of basis $\phi_n$. The above equation implicitly assumes a linear ordering of the elements in $P$:
\begin{align*}
  \left\{ \phi_{\alpha}\right\}_{\alpha \in \Lambda} \Longleftrightarrow \left\{ \phi_n\right\}_{n=1}^N
\end{align*}
We will occasionally make use of this identification for notational convenience; the ordering of the indices in $\Lambda$ with respect to $1, \ldots, N$ is irrelevant in our context.

\rev{In this paper we will consider general index sets $\Lambda$, but some of our theoretical results focus on the multi-index set $\Lambda_k \triangleq \left\{ \alpha \in \N_0^d\, | \, |\alpha| \leq k\right\}$ corresponding to the degree-$k$ polynomial space. We will use the notation $P_k$, $\Pi_k$, and $K_k$ in the special case $\Lambda = \Lambda_k$. See Table \ref{tab:notation} for a summary of notation.}



\subsection{Discrete least-squares approximation}\label{sec:setup-dls}
We consider the problem of least-squares approximation using discrete collocation samples onto a polynomial space $P$ defined by general index set $\Lambda$. This regression problem is a discrete approximation to a continuous projection. We approximate the $L^2_w$-orthogonal projection of a function $f(z)$ onto $P$ by sampling $f$ at discrete locations. 

The continuous projection onto $P$ satisfies 
\begin{align}\label{eq:continuous-projection}
  \Pi f \triangleq \argmin_{p \in P} \left\| f - p \right\|_{L^2_w} = \argmin_{p \in P} \E \left[ f(z) - p(z) \right]^2,
\end{align}
where in the latter equality we consider $z$ a random variable with density $w$. In practice this optimal projection can rarely be computed because of insufficient knowledge about $f$. An alternative approach is discrete approximation: compute the minimizer of a discretization of the continuous norm. If $\left\{z_s\right\}$ are iid samples of the random variable $z$, then an approximation to $\Pi f$ from \eqref{eq:continuous-projection} can be computed by using $S$ of these samples:
\begin{align}\label{eq:unweighted-least-squares}
  \Pi^S f = \argmin_{p \in P} \frac{1}{S} \sum_{s=1}^S \left| f(z_s) - p(z_s) \right|^2.
\end{align}
If the samples $z_s$ are drawn iid from $w$, then it is straightforward to see that $\lim_{S \rightarrow \infty} \Pi^S f = \Pi f$, with more precise conditions on accuracy in \cite{cohen_stability_2013}, whose main results are reiterated in Section \ref{sec:background-ls}. More generally, one can take $\left\{z_s\right\}$ to be iid samples, but drawn from a different density $v$, in which case a weighted formulation is required to approximate the $w$ norm
\begin{align}\label{eq:weighted-least-squares}
  \widetilde{\Pi}^S f = \argmin_{p \in P} \frac{1}{S} \sum_{s=1}^S k_s \left| f(z_s) - p(z_s) \right|^2 
\end{align}
Using a change-of-measure argument, we see that if we choose $k_s = w(z_s)/v(z_s)$ and the support of $v$ contains the support of $w$, then $\lim_{S \rightarrow \infty} \widetilde{\Pi}^S f = \Pi f$. 

In either case \eqref{eq:unweighted-least-squares} or \eqref{eq:weighted-least-squares}, if the $S$ samples $z_s$ are given, we can formulate the algebraic version of these problems. For a fixed index set $\Lambda$, recall that $N = N(\Lambda)$ denotes the dimension of $P$. Let $\mathbf{V}$ be the $S \times N$ Vandermonde-like matrix for the basis $\phi_n$ with samples $z_s$: $(V)_{s,n} = \phi_n(z_s)$. We may express the approximation $\widetilde{\Pi}^S$ from \eqref{eq:weighted-least-squares} in the basis $\phi_\alpha$:
\begin{align*}
  \widetilde{\Pi}^S f = \sum_{\alpha \in \Lambda} c_\alpha \phi_\alpha(z).
\end{align*}
We collect the unknown coefficients $c_\alpha$ into the vector $\mathbf{c}$, and the function evaluations $f(z_s)$ into the vector $\mathbf{f} \in \R^S$. The solution to \eqref{eq:weighted-least-squares} is defined by the least-squares solution to the following $k_s$-weighted problem:
\begin{align}\label{eq:regression-algebraic-minimization}
  \mathbf{c} = \argmin_{\mathbf{g} \in \R^N} \left\| \sqrt{\mathbf{K}} \mathbf{V} \mathbf{g} - \sqrt{\mathbf{K}} \mathbf{f} \right\|^2,
\end{align}
where $\mathbf{K}$ is an $S \times S$ diagonal matrix with entries $(K)_{s,s} = k_s$. Equivalently, we may seek the solution to the normal equations:
\begin{align}\label{eq:normal-equations}
  \mathbf{G} \mathbf{c} = \frac{1}{S} \mathbf{V}^T \mathbf{K} \mathbf{f},
\end{align}
where $\mathbf{G}$ is an $N \times N$ Gramian matrix with the random entries
\begin{align}\label{eq:discrete-inner-product}
  (G)_{n,m} &= \left\langle \phi_n, \phi_m \right\rangle_{S}, & 
  \left\langle g, h \right\rangle_{S} \triangleq \frac{1}{S} \sum_{s=1}^S k_s g(z_s) h(z_s).
\end{align}
In this paper, we seek to specify the measure from which the $z_s$ are drawn, and subsequently the weights $k_s$. We will see that in the CLS algorithm, $k_s \approx v(z_s)/w(z_s)$, i.e., that our change of measure is not exactly faithful to $w$. The unweighted algorithm associated with \eqref{eq:unweighted-least-squares}, proceeding by choosing $k_s \equiv 1$ and taking the sampling density equal to orthogonality density $v = w$, is given in Algorithm \ref{alg:unweighted-mc}.

\input{algorithms/unweighted-ls}

\rev{
\begin{table}
  \begin{center}
  \resizebox{\textwidth}{!}{
    \renewcommand{\tabcolsep}{0.4cm}
    \renewcommand{\arraystretch}{1.3}
    {\scriptsize
    \begin{tabular}{@{}cp{0.8\textwidth}@{}}\toprule
      Symbol(s) & \\ \midrule
      $\phi_\alpha$, $\phi_n$ & $L^2_w$-orthonormal polynomials. $\deg \phi_\alpha = |\alpha|$. A bijection between $\alpha \in \Lambda$ and $n \in \left\{1, 2, \ldots, N\right\}$ is assumed. \\
      $\Lambda$, $N$ & General multi-index set, a subset of $\N_0^d$. $N = \left|\Lambda\right|$ \\
      $\Lambda_k$ & The degree-$k$ total-degree multi-index set: $\Lambda_k = \left\{ \alpha \in \N_0^d\, | \, |\alpha| \leq k\right\}$ \\
      $P$ ($P_k$) & Polynomial space spanned by $\phi_\alpha$ for all $\alpha \in \Lambda$. ($P_k$ corresponds to $\Lambda = \Lambda_k$)\\
      $\Pi$ ($\Pi_k$) & $L^2_w$-orthogonal projector onto $P$ (respectively, onto $P_k$)\\
      $K$ ($K_k$) & Reproducing kernel ``diagonal" of $P$ in $L^2_w$. (respectively, of $P_k$ in $L^2_w$)\\
      $\widetilde{w}$ & Equals equilibrium density weighted by normalized Christoffel function, see Section \ref{sec:cls-theory} \\
      $\widetilde{\Pi}$ ($\widetilde{\Pi}_k$) & $L^2_{\widetilde{w}}$-orthogonal projector onto $P$ (respectively, onto $P_k$)\\
      $\V{R}$ ($\V{R}_k$) & $L^2_{\widetilde{w}}$ Gramian of $\phi_\alpha$ for $\alpha \in \Lambda$ (respectively, for $\alpha \in \Lambda_k$)\\
    \bottomrule
    \end{tabular}
  }
    \renewcommand{\arraystretch}{1}
    \renewcommand{\tabcolsep}{12pt}
  }
  \end{center}
  \caption{\rev{Notation used throughout this article.}}\label{tab:notation}
\end{table}
}

\subsection{Equilibrium measures}
\rev{We review some results in weighted pluripotential theory with the goal of introducing the (weighted) equilibrium measure. In the CLS algorithm, this measure will define the sampling density $v$ from Section \ref{sec:setup-dls}. Standard references for pluripotential theory are \cite{bloom_weighted_2009,klimeck_pluripotential_1991} and Appendix B of \cite{saff_logarithmic_1997}. }

In the following brief discussion of pluripotential theory, we need to define the weight function $\exp(-Q)$ that is derived from the function $Q$. The weight function $\exp(-Q(z))$ on $D$ serves as the ``pluripotential weight function". In general our discussion in this subsection does not require $\exp(-Q)$ to be related to the orthogonality weight function $w$. However, in the context of the CLS algorithm, we will identify the pluripotential weight function $\exp(-Q)$ with $\sqrt{w}$ for the CLS-unbounded case, whereas in the CLS-bounded case we will always take $Q=0$, and thus $\exp(-Q) \equiv 1$.

\rev{
Consider the class of plurisubharmonic functions on $\C^d$ that grow at most logarithmically at infinity:
\begin{align*}
  \mathcal{L} = \left\{ u \textrm{ plurisubharmonic on } \C^d \; \big| \; u(z) \leq \max\left\{ \log(z), 0 \right\} + C \right\},
\end{align*}
where $C$ is a constant that depends on $u$. Given a pluripotential-theoretic-admissible domain $D$ and weight $\exp(-Q(z))$, the \textit{weighted extremal function} is the function
\begin{align*}
  V_{D,Q}(z) = \left\{ \sup u(z) \; \big|\;  u \in \mathcal{L} \textrm{ satisfying } u \leq Q \textrm{ on } D\right\}
\end{align*}
This function may be represented as an upper envelope of the logarithm of polynomials, whence the connection to polynomial approximation can be established. The regularization of this function defined as $V^\ast_{D,Q}(z) = \limsup_{\xi \rightarrow z} V_{D,Q}(\xi)$ is an uppersemicontinuous function. The weighted pluripotential equilibrium measure is given by
\begin{align*}
  \mu_{D,Q} = \frac{1}{(2 \pi)^d} \left( d d^c V^{\ast}_{D,Q} \right)^d,
\end{align*}
where $\left(d d^c u\right)^d$ is the complex Monge-Amp{\`e}re operator applied to $u$. (In the previous equation the superscript $d$ is the integer dimension, whereas the normal-text $d$ is a complex differential operator.) The measure $\mu_{D,Q}$ is a probability measure and has compact support. Sharp conditions under which $\mu_{D,Q}$ is absolutely continuous with respect to Lebesgue measure are not known in general, but some sufficient conditions are given in \cite{bedford_complex_1986,bloom_weighted_2009}.

In one dimension, these concepts reduce to one-dimensional concepts from potential theory: $V_{D,Q}$ is the weighted complex Green's function, $d d^c$ is (proportional to) the complex Laplacian, and $\mu_{D,Q}$ is the potential-theoretic (weighted) equilibrium measure.
}

For the bounded CLS-admissible case, we need the unweighted ($Q \equiv 0$) measure $\mu_D \equiv \mu_{D,0}$; we will denote its Lebesgue density as $\dx{\mu_D}(z) = v_D(z) \dx{z}$ (assuming such a density exists). The bounded CLS-admissible assumptions above guarantee in this case that $D$ is potential-theoretic admissible and so $\mu_D$ is well-defined, and is a probability measure on $D$. 

For the unbounded CLS-admissible case, we will take $\exp(-Q) = \sqrt{w}$ so that the potential-theoretic quantity $Q$ is given by $Q = - \frac{1}{2} \log w$. We denote the Lebesgue density of the corresponding weighted equilibrium measure as $\dx{\mu_{D,Q}}(z) = v_{D,Q}(z) \dx{z}$ (assuming such a density exists). The weighted equilibrium measure $\mu_{D,Q}$ here has compact support even though $D$ is unbounded. 


With regards to the CLS algorithm, we use $\mu_{D}$ as the sampling measure for the bounded case, and a scaled version of $\mu_{D,Q}$ as a sampling measure for the unbounded case. Section \ref{sec:cls} gives examples of the equilibrium measure density $v_D$. 

%% file: algorithms/unweighted-ls.tex
\begin{algorithm}
\SetKwInOut{Input}{input}\SetKwInOut{Output}{output}

\Input{Weight/density function $w$ with associated orthonormal family $\phi_\alpha$, index set $\Lambda$, function $f$}
\Output{Expansion coefficients $\V{c}$ to approximate $\Pi_\Lambda f$}

\BlankLine
Generate $S$ iid samples $\left\{z_s\right\}$ from density $w$\;
Assemble $\V{f}$ with entries $(f)_s = f(z_s)$\;
Form $S \times N(\Lambda)$ Vandermonde-like matrix $\V{V}$ with entries $(V)_{s,n} = \phi_{\alpha(n)}(z_s)$\;
Compute $\V{c} = \argmin_{\V{g} \in \R^N} \left\| \V{V} \V{g} - \V{f} \right\|$\;

\caption{Unweighted least squares with Monte Carlo (LSMC)}\label{alg:unweighted-mc}
\end{algorithm}

%% file: content/cls.tex
\section{Christoffel Least Squares}\label{sec:cls}

This section describes the novel algorithmic content of this paper: the Christoffel Least-Squares (CLS) algorithm applied to Monte Carlo approximation of $L^2$ projections. Essentially, this algorithm solves the problem \eqref{eq:weighted-least-squares} and thus requires (a) specification of the weights $k_s$ (i.e. the matrix $\V{K}$) and (b) specification of the sampling measure $v(z)$. The CLS algorithm takes on different formulations when the domain $D$ is bounded or unbounded, but a common formula in both cases is the specification of the weights $k_s$. For a general index set $\Lambda$, the CLS algorithm chooses $k_s$ to be quantities that scale each row of $\sqrt{\V{K}} \V{V}$ to have $\ell^2$ norm equal to the constant $N$, i.e.,
\begin{align}\label{eq:christoffel-weights}
  k_s = \frac{N}{\sum_{\alpha \in \Lambda} \phi_\alpha^2(z_s)} = \left(\frac{K(z_s)}{|\Lambda|}\right)^{-1}.
\end{align}
Thus, the weights $k_s$ are evaluations of the normalized Christoffel function.

When $D$ is bounded, the CLS algorithm chooses the sampling weight function as $v = v_D$, the density function of the (unweighted) pluripotential equilibrium measure of $D$. When $D$ is an unbounded conic domain, the sampling weight function $v$ is a scaled version of the $\sqrt{w}$-\textit{weighted} pluripotential equilibrium measure of the domain $D$. Thus, the particular specification of the sampling measure differs when $D$ is bounded versus unbounded. 

Informally, the CLS algorithm is reasonable because it adheres to the change-of-measure argument following equation \eqref{eq:weighted-least-squares}: if $v$ is the suitable equilibrium measure density and we consider approximation with $\Lambda = \Lambda_k$, then 
\begin{align}\label{eq:approx-christoffel-convergence}
  k_s = \frac{N}{K_{k}(z_s)} \sim \frac{w(z_s)}{v(z_s)}.
\end{align}
Indeed, this is true for very general weights and domains and is a major result in weighted pluripotential theory, which we discuss in Sections \ref{sec:background} and \ref{sec:cls-theory}. We also observe in Section \ref{sec:results} that, even for non-total-degree spaces, the CLS approximation performs quite well.

\subsection{Bounded domains}\label{sec:cls-bounded}

Let $w(z)$ be an admissible weight function on a compact domain $D$. In this case we sample with the density given by the (unweighted) equilibrium measure
\begin{align}\label{eq:cls-bounded-weight}
  v(z) = v_D(z) = \dfdx{\mu_{D}}{z},
\end{align}
which is a probability measure. (For our discussion, we assume $v_D$ exists.) We emphasize that this weight $v$ is independent of the orthogonality density $w$. The weights $k_s$ are as given in \eqref{eq:christoffel-weights}, and indirectly build in the dependence on $w$ through \eqref{eq:approx-christoffel-convergence}. The method is shown in Algorithm \ref{alg:cls-bounded}. 

\input{algorithms/cls-bounded}


The pluripotential equilibrium measure generalizes the univariate potential-theoretic measure. In one dimension on $D = [-1,1]$, the measure $\mu_D$ is the arcsine measure with ``Chebyshev" density $v_D(z) = \frac{1}{\pi\sqrt{1-z^2}}$. Thus, on an interval the CLS algorithm prescribes Chebyshev sampling regardless of the weight. This conforms with the colloquially well-known observation that the Chebyshev measure on an interval is somehow ``universal", e.g., \cite{rauhut_sparse_2012}.

The equilibrium measure for $D = [-1,1]^d$ is the product measure of the univariate measure. For more complicated multivariate domains, computing $v_D(z)$ is not trivial, but some special cases have explicit formulas. For example, if $D$ is the unit ball in $\R^d$, $D = \left\{ z \in \R^d \, \big| \, \|z\|_{2} \leq 1\right\}$, the equilibrium measure $\mu_D$ has density $v_D(z) = \frac{2}{V_d} \left[1 - \|z\|_{2}\right]^{-1/2}$ with $V_d$ the volume of $D$ \cite{bos_asymptotics_1994}. Analysis for more general convex, origin-symmetric domains is given in \cite{bedford_complex_1986}. When $D$ is the unit simplex, $D = \left\{z\in \R^d \, \big| \, z_j \geq 0, \;\; 1 - \left\|z\right\|_{1} \geq 0\right\}$, the equilibrium density is $v_D(z) = C \left[ \left(1 - \left\|z\right\|_{1}\right) \prod_{j=1}^d z_j\right]^{-1/2}$, with $C = \pi^{(d+1)/2}/\Gamma((d+1)/2)$, see, e.g., \cite{xu_asymptotics_1999}. A fairly general result for convex sets is given in \cite{burns_monge-ampere_2010}.

Table \ref{tab:cls-sampling} summarizes some formulas for $v_{D,Q}$. In that table, the following notation for sets is used:
\begin{subequations}\label{eq:ball-simplex}
\begin{align}
  B^d &= \left\{ z = (z_1, \ldots, z_d) \in \R^d\; \big| \; |z|^2 = \sum_{j=1}^d z_j^2 \leq 1 \right\} \\
  T^d &= \left\{ z = (z_1, \ldots, z_d) \in \R^d\; \big| \; z_j \geq 0 \textrm{ and } \|z\|_1 = \sum_{j=1}^d z_j \leq 1 \right\}
\end{align}
\end{subequations}

\subsection{Unbounded domains}\label{sec:cls-unbounded}
We now consider the case of unbounded $D \subset \R^d$, for which we recall the assumption \eqref{eq:q-homogeneous}, that $Q = -\frac{1}{2} \log w$ is a homogeneous function of order $t$.

The CLS algorithm in this unbounded case chooses the weights as in \eqref{eq:christoffel-weights}, but the sampling weight function is a scaled version of the $\sqrt{w}$-weighted equilibrium measure.
\begin{align}\label{eq:cls-unbounded-weight}
  v(z) = k^{-d/t} v_{D,Q}\left(k^{-1/t} z\right) &= k^{-d/t} \dfdx{\mu_{D,Q}}{z}\left(k^{-1/t} z\right), & k &= \max_{\alpha \in \Lambda} |\alpha|.
\end{align}
Again, for the purposes of discussion we assume that the Lebesgue density $v_{D,Q}$ exists. The measure $\mu_{D,Q}$ has compact support, which is why we require scaling by $k^{1/t}$, effectively expanding the support to contain the ``important" parts of the domain. We give the method in Algorithm \ref{alg:cls-unbounded}. In the unbounded case, the scaling depends on the index set $\Lambda$ defining the polynomial space $P$.

To our knowledge, explicit formulae for multivariate \textit{weighted} equilibrium measures $\mu_{D,Q}$ on real-valued sets $D$ are currently unknown for many real-valued sets $D$, even the ``canonical" ones considered here. This is the case even for the simple case $\sqrt{w} = \exp(-|z|^2)$ so that $Q(z) = |z|^2$ on $D = \R^d$. However, we \textit{conjecture} the following density for the equilibrium measure in this case:
\begin{align*}
  \left.\begin{array}{l}
    Q(z) = \sqrt{w(z)} = \exp(-|z|^2), \\
    D = \R^d
  \end{array}\right\}
  \Longrightarrow
  \dfdx{\mu_{D,Q}}{z} = v_{D,Q}(z) = C \left[ 1 - |z|^2 \right]^{d/2},
\end{align*}
where $C$ is a normalization constant, and $\mu_{D,Q}$ is supported only on the set $B^d$. Similarly, for the weight $\sqrt{w} = \exp\left(-\sum_{j=1}^d z_j\right)$, we \textit{conjecture} the following equilibrium measure density:
\begin{align*}
  \left.\begin{array}{l}
    Q(z) = \sqrt{w(z)} = \exp\left(-\sum_{j=1}^d z_j\right), \\
    D = [0, \infty)^d \subset \R^d
  \end{array}\right\}
  \Longrightarrow
  \dfdx{\mu_{D,Q}}{z} = v_{D,Q}(z) = C \sqrt{\frac{\revision{\left( 2 - \sum_{j=1}^d z_j\right)^d}}{\prod_{j=1}^d z_j}},
\end{align*}
where again $C$ is a normalization constant, $\mu_{D,Q}$ is supported only on points $z \in 2 T^d$. In general, even computing just the support of $\mu_{D,Q}$ is a nontrivial task \cite{alan_supports_2013}. Note, however, in $d=1$ dimension, the weighted equilibrium measure is explicitly known for a wide class of weights on bounded and unbounded real-valued sets, e.g., \cite{saff_logarithmic_1997}. Our conjectured densities for $d \geq 1$ above specialize with $d=1$ to these known densities.

Numerical experiments we have conducted support our conjectures above, and examples shown in Section \ref{sec:results} are generated using these sampling schemes and show very good performance, which further support our conjectures. Section \ref{sec:results} also gives methodology for sampling from our conjectured densities.
 
\input{algorithms/cls-unbounded}


\begin{table}
  \begin{center}
  \resizebox{\textwidth}{!}{
  \renewcommand{\tabcolsep}{0.4cm}
  \renewcommand{\arraystretch}{1.8}
  \begin{tabular}{@{}cccc@{}}\toprule
    Domain $D$ & Orthogonality weight $w$ & Sampling density domain & Sampling density $v(y) = \dfdx{\mu_{D,Q}}{y}$ \\ \midrule
    $[-1, 1]^d$ & Any admissible weight & $[-1,1]^d$ & $\frac{1}{\pi\prod_{j=1}^d \sqrt{1 - y_j^2}}$ \\
    $B^d$ & Any admissible weight & $B^d$ & $\frac{C}{\sqrt{1 - |y|^2}}$ \\[5pt]
    $T^d$ & Any admissible weight & $T^d$ & $C \sqrt{\frac{1 - \sum_{j=1}^d y_j}{\prod_{j=1}^d y_j}}$ \\
    $\R^d$ & $\exp(-|z|^2)$ & $(\ast\ast)\sqrt{2} B^d$ & $(\ast\ast)C \left[ 2 - |y|^2 \right]^{d/2}$ \\
    $[0,\infty)^d$ & $\exp\left(-\sum_{j=1}^d z_j\right)$ & $(\ast\ast)4 T^d$ & $(\ast\ast)C \sqrt{\frac{\revision{\left(4 - \sum_{j=1}^d y_j\right)^d}}{\prod_{j=1}^d y_j}}$ \\
  \bottomrule
  \end{tabular}
  \renewcommand{\arraystretch}{1}
  \renewcommand{\tabcolsep}{12pt}
}
\end{center}
\caption{Explicit CLS sampling strategies for particular bounded and unbounded scenarios. The sets $B^d$ and $T^d$ are defined in \eqref{eq:ball-simplex}. Entries preceded with $(\ast\ast)$ are \textit{conjectures} only. (Our conjectures specialize to known, correct results in one dimension \cite{saff_logarithmic_1997}.) For unbounded domains, the sampling density shown is over a compact domain; degree-scaled sampling according to Algorithm \ref{alg:cls-unbounded} should be performed, which depends on $\Lambda$. 
}\label{tab:cls-sampling}
\end{table}

We have completed specification of the CLS algorithm. It is a straightforward weighted Monte Carlo approach if, given $D$ and $w$, the equilibrium measure $v_D$ or $v_{D,Q}$ is known and may be sampled from. Our analysis presented later indicates that $v$ is the \textit{asymptotically} optimal measure, but is not strictly optimal for a fixed $\Lambda$. In general, one could consider sampling instead from the measure $\widetilde{w}(z) = w(z) \frac{K(z)}{N}$, and this is briefly explored in \cite{hampton_coherence_2014}.


%% file: algorithms/cls-bounded.tex
\begin{algorithm}
\SetKwInOut{Input}{input}\SetKwInOut{Output}{output}

\Input{Weight/density function $w$ with associated orthonormal family $p_\alpha$, polynomial index set $\Lambda$ and corresponding size $N$, function $f$}
\Output{Expansion coefficients $\V{c}$ to approximate $\Pi u$}

\BlankLine
Generate $S$ iid samples $\left\{z_s\right\}$ from equilibrium measure $\mu_D$\;
Assemble $\V{u}$ with entries $(u)_s = u(z_s)$\;
Compute LS weights $\V{K}$ with entries $(K)_{s,s} = N/{K(z_s)}$ from \eqref{eq:christoffel-weights}\;
Form $S \times N$ Vandermonde-like matrix $\V{V}$ with entries $(V)_{s,n} = \phi_{n}(z_s)$\;
Compute $\V{c} = \argmin_{\V{g} \in \R^N} \left\| \sqrt{\V{K}} \V{V} \V{g} - \sqrt{\V{K}} \V{u} \right\|$\;

\caption{Christoffel Least Squares (CLS) on a compact domain $D$}\label{alg:cls-bounded}
\end{algorithm}

%% file: algorithms/cls-unbounded.tex
\begin{algorithm}
\SetKwInOut{Input}{input}\SetKwInOut{Output}{output}

\Input{Weight/density function $w$ with associated orthonormal family $p_\alpha$, index set $\Lambda$ and corresponding size $N$, function $f$}
\Output{Expansion coefficients $\V{c}$ to approximate $\Pi u$}

\BlankLine
Compute log-weight homogeneity factor $t$ from a $w = \exp(-2 Q)$ identification in \eqref{eq:q-homogeneous}\;
Generate $S$ iid samples $\left\{z_s\right\}$ from equilibrium measure $\mu_{D,Q}$\;
Expand samples: $z_s \gets k^{1/t} z_s$, where $k = \max_{\alpha \in \Lambda} |\alpha|$\;
Assemble $\V{f}$ with entries $(f)_s = f(z_s)$\;
Compute LS weights $\V{K}$ with entries $(K)_{s,s} = N/{K(z_s)}$ from \eqref{eq:christoffel-weights}\;
Form $S \times N(\Lambda)$ Vandermonde-like matrix $\V{V}$ with entries $(V)_{s,n} = \phi_{\alpha(n)}(z_s)$\;
Compute $\V{c} = \argmin_{\V{g} \in \R^N} \left\| \sqrt{\V{K}} \V{V} \V{g} - \sqrt{\V{K}} \V{f} \right\|$\;

\caption{Christoffel Least Squares (CLS) on an unbounded $D$}\label{alg:cls-unbounded}
\end{algorithm}

%% file: content/background.tex
\section{Background}\label{sec:background}
This section recalls the two cornerstone results we require: general discrete least-squares stability and accuracy from \cite{cohen_stability_2013}, and asymptotics of the Christoffel function from \cite{berman_bergman_2009,berman_bergman_2009-1}.

\subsection{Least squares stability and convergence}\label{sec:background-ls}
This section summarizes the main results of \cite{cohen_stability_2013}. For a weight function $w$ on $D \subset \R^d$, let $\phi_n$ for $n=1,2,\ldots$ be any $w$-orthonormal system (not necessarily polynomials). To be consistent with previous notation, we let $P$ denote the subspace spanned by these elements, even though this need not be a polynomial space. The diagonal of the reproducing kernel $K(z)$ is as before in \eqref{eq:K-def} and, with $P$ fixed, does not depend on the particular choice of basis. 

If we sample $S$ iid realizations according to the weight $w$ and form a discrete least-squares problem, the $S$-dependent Gramian matrix $\V{G}$ in \eqref{eq:discrete-inner-product} satisfies $\lim_{S \rightarrow \infty} \V{G} = \V{I}$, with $\V{I}$ the $N \times N$ identity matrix. $\V{G}$ is a random matrix, and one can use random matrix estimates to precisely specify the probability with which $\V{G}$ is close to $\V{I}$. In turn, this yields estimates on accuracy of the least-squares solution. In the following, $\vertiii{\cdot}$ is the spectral norm of a matrix.
\begin{theorem}[\cite{cohen_stability_2013}]\label{thm:ls-stability}
  Let $P$ be any $N$-dimensional subspace of $L^2_w$. Assume that $\left\{z_s\right\}_{s=1}^S$ are $S$ iid samples drawn from the density $w$, and that the number of samples satisfies
  \begin{align}\label{eq:cohen-sample-size}
    \frac{S}{N \log S} \geq \left[\frac{1+r}{c_\delta}\right] \frac{\|K(z)\|_\infty}{N}
  \end{align}
  with $c_\delta \triangleq \delta + (1 - \delta) \log (1 - \delta)$ for some $\delta \in (0,1)$ and $r > 0$. Then the discrete Gramian matrix $\mathbf{G}$ given by \eqref{eq:discrete-inner-product} satisfies the following stability condition:
  \begin{align}\label{eq:cohen-stability}
    \mathrm{Pr} \left[\, \vertiii{ \mathbf{G} - \mathbf{I}} > \delta \,\right] \leq \frac{2}{S^r}.
  \end{align}
  Furthermore, letting $\Pi^S$ be the unweighted Monte Carlo projection operator defined in \eqref{eq:unweighted-least-squares}, then the following convergence result holds for any $f \in L^2_w$ satisfying $\|f\|_\infty \leq L$:
  \begin{align}\label{eq:cohen-convergence}
    \E \left\| f - T_L\left(\Pi^S f \right)\right\|_w^2 \leq \left[ 1 + \varepsilon(S) \right] \left\|g\right\|_w^2 + \frac{8 L^2}{S^r},
  \end{align}
  where $g = (I - \Pi) f$ so that $\|g\|_w$ is the $L^2_w$ optimal error, $\varepsilon(S) \triangleq \frac{4 c_\delta}{(1+r) \log S}$, and $T_L(x)= \mathrm{sgn} (x) \min\left\{|x|, L\right\}$ is a truncation function.
\end{theorem}
The critical term in the ensemble size condition \eqref{eq:cohen-sample-size} is the maximum of the reproducing kernel diagonal, $\|K\|_\infty/N$. Clearly this quantity depends on (i) the weight $w$ and (ii) the space $P$. 

When $D = [-1,1]$ and $w=1/\sqrt{1-z^2}$ with $P$ the degree-$n$ polynomial space, a system of orthonormal polynomials is given by the Chebyshev polynomials $\phi_n(z) = T_n(z)$. The reproducing kernel diagonal for this one-dimensional basis satisfies
\begin{align}\label{eq:chebyshev-christoffel}
  \|K_k(z)\|_{\infty} = \sup_{z \in D} \sum_{j=0}^{k} T_j(z)^2 \equiv 2 k + 1 = 2 N - 1
\end{align}
for all $k \in \N$. Thus, so long as $S / (N \log S) \gtrsim C$ for an absolute constant $C$, then one can obtain stability from \eqref{eq:cohen-stability}. In Section \ref{sec:results} showing our numerical experiments, we call the similar scaling $S \sim N \log N$ ``log-linear" scaling of $S$ with respect to $N$.

This optimal $\|K\|_{\infty} \sim N$ behavior only happens for very special weights. Consider the interval $D = [-1,1]$ with weight function $w(z) = (1-z^2)^\beta$ for $\beta \geq 0$ with corresponding orthonormal polynomial family $\phi_n$. An understanding of how the choice of $\beta$ affects the sampling criterion can be communicated by Figure \ref{fig:maxK-plots}. We show $\|K_k\|_\infty/N$ for various combinations of degree $k$ and $\beta$ values, with the conclusion that this quantity becomes extremely large when $k$ or $\beta$ is increased, and is very large even for moderate values of these parameters. Therefore, the restriction on the number of samples required for stability becomes onerous even on bounded sets in one dimension if one considers non-Chebyshev weights. \rev{Theoretical upper bounds for $\|K\|_\infty$ in the (tensor-product) multivariate case appear in \cite{migliorati_multivariate_2015}; their univariate behavior matches that shown in Figure \ref{fig:maxK-plots}.}

\begin{figure}
\resizebox{\textwidth}{!}{
\includegraphics{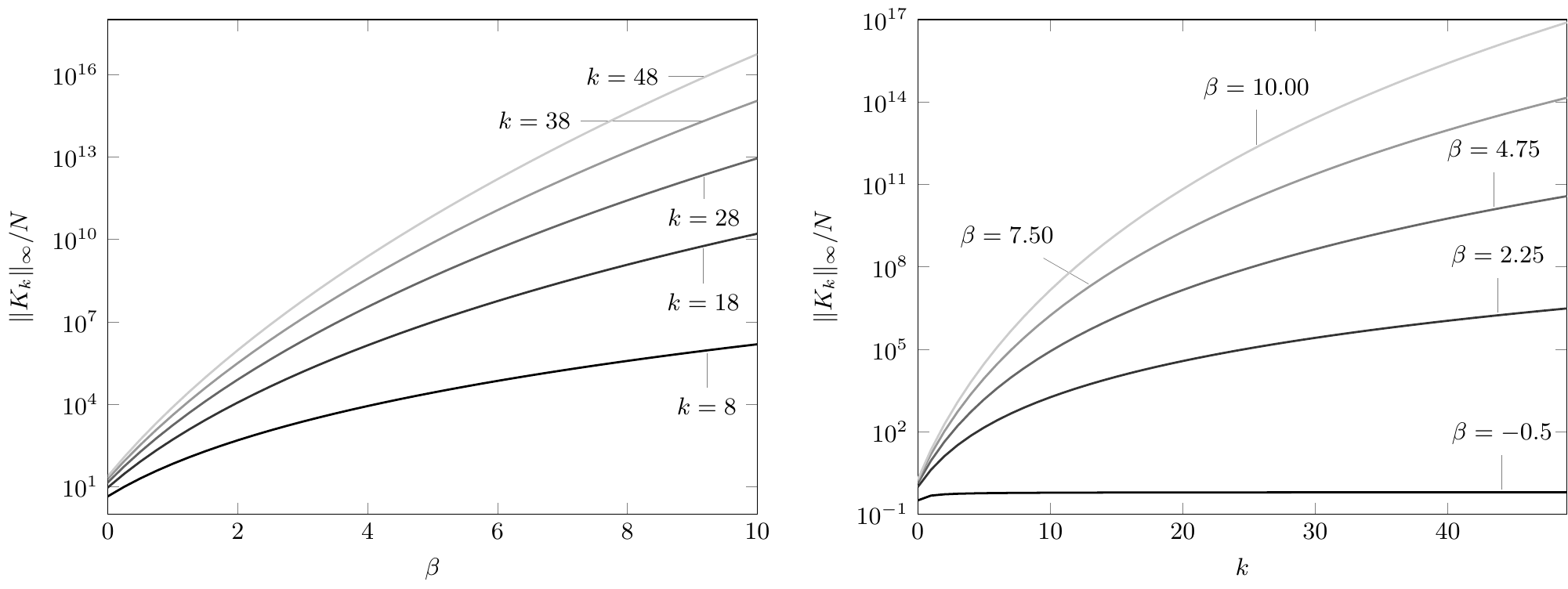}
}
\caption{Values of the stability factor $\|K_k\|_\infty/N$ with $N = \dim P_k = k+1$ for one-dimensional symmetric Jacobi polynomial families (parameters $\alpha = \beta$). Left: plots versus parameter $\beta$ for various degrees $k$. Right: plots versus degree $k$ for various parameters $\beta$.}\label{fig:maxK-plots}
\end{figure}

Given the observations above, a standard MC approach is not tractable for all weights using this analysis, and in particular for polynomials on unbounded domains such as Hermite polynomials, one cannot directly use the above analysis if $K$ is a polynomial kernel. (A straightforward remedy would be to instead use weighted polynomials, which is effectively what the CLS algorithm does.)

\subsection{Christoffel function asymptotics}\label{sec:background-equilibrium}

For polynomials, the reproducing kernel diagonal quantity $K(z)$ is the inverse of the Christoffel function from the theory of orthogonal polynomials, about which much is known. The purpose of this section is to review formal results establishing the behavior in \eqref{eq:approx-christoffel-convergence}. 

For unbounded domains we will need to discuss asymptotics of measures whose mass escapes to infinity; to make mathematically sensible statements in this case we adopt the strategy of ``compressing" these measures by a scaling factor so that their mass remains on a compact interval. One way to implement this compression is to frame the discussion with respect to a \textit{varying} weight $w^k$, i.e., $w$ raised to the power $k$. The effect of using varying weights is that the measures we are interested in place their support on compact domains.

On unbounded domains $D$ with a pluripotential-theoretic weight function $\rho = \exp(-Q)$ given, we consider polynomials orthonormal under the \textit{varying} weight function $\rho^{2 k}$ for an integer $k$. For each $k \in \N$, let $\phi^{(k)}_\alpha$ be the family orthonormal under $\rho^{2 k}$, i.e.,
\begin{align}\label{eq:phi-varying-weight}
  \int_D \phi^{(k)}_\alpha(z) \phi^{(k)}_\beta(z) \rho^{2 k}(z) \dx{z} = \delta_{\alpha,\beta}.
\end{align}
Then the reproducing kernel diagonal associated to $P_k$ with the varying weight function $\rho^{2 k}$ is given by
\begin{align*}
  K^{(k)}_k = \sum_{|\alpha| \leq k} \left(\phi^{(k)}_\alpha(z)\right)^2
\end{align*}
Note that $K^{(1)}_k \equiv K_k$, the standard $L^2_w$-kernel diagonal of $P_k$.

Asymptotics for both $K$ (on bounded domains) and $K_k^{(k)}$ (unbounded domains) are known in many cases. In one dimension on $D = [-1,1]$, the relation
\begin{align}\label{eq:1d-convergence-bounded}
  \lim_{k \rightarrow \infty} \frac{K_k(z)}{N} = \frac{1}{\pi w(z) \sqrt{1-z^2}}
\end{align}
holds for almost every $z \in D$, for any $w$ that is bounded and continuous. Similar results for $K_k^{(k)}$ on unbounded domains hold. (See, e.g., \cite{bloom_asymptotics_2008,totik_asymptotics_2000,totik_asymptotics_2000-1,nevai_geza_1986}.) Indeed, for more general $w$ and multidimensional $D$, the above result holds if one replaces $\frac{1}{\pi \sqrt{1-z^2}}$ above by the Lebesgue density of the pluripotential equilibrium measure \cite{bos_asymptotics_1994,xu_christoffel_1995,xu_asymptotics_1999,bos_asymptotics_1998,kroo_christoffel_2013}.

Since $K$ is a specialization of $K^{(k)}$ when $Q \equiv 0$, it is sufficient to consider asymptotics of $K^{(k)}$, which is the goal of the following general result.
\begin{theorem}[\cite{berman_bergman_2009,berman_bergman_2009-1,berman_fekete_2011}]\label{thm:christoffel-convergence}
  Let $D$ be a potential-theoretic admissible domain in $\R^d$ equipped with a smooth weight $\dx{V}(z) = q(z) \dx{z}$, along with a bounded and continuous weight function $\rho$ such that $\rho(z) \dx{V}(z)$ defines an orthonormal polynomial family $\phi_\alpha$ in $L^2_{q \rho}(D)$. With $K_k^{(k)}$ the $L^2_{q \rho^{2 k}}$ reproducing kernel diagonal of the total-degree polynomial space $P_k$, then the following convergence holds weakly:
  \begin{align*}
    \lim_{k \rightarrow \infty} \frac{1}{N} \rho^{2 k}(z) K^{(k)}_k(z) \dx{V}(z) = \dx{\mu_{D,Q}}(z).
  \end{align*}
\end{theorem}
Note that the above is a very general result, but special cases of this (especially in one dimension) have been known long before the references cited in the Theorem: See., e.g., \cite{totik_asymptotics_2000} and references therein. We will specialize the theorem above to two cases: (i) on bounded $D$, we set $\rho \equiv 1$ and $\dx{V}(z) = w(z) \dx{z}$, (ii) on unbounded $D$ with $w = \exp(-2 Q)$, we set $\rho = \exp(-Q)$ and $\dx{V}(z) = \dx{z}$. 

\begin{corollary}\label{cor:christoffel-convergence}
  The following two cases are specializations of Theorem \ref{thm:christoffel-convergence}. For both cases below, we assume that the density $v_{D,Q}$ exists.
  \begin{enumerate}
    \item Let $w$ be any continuous weight function admitting an orthogonal polynomial basis on a compact, connected set $D$. With $\rho \equiv 1$ and $\dx{V} = w(z) \dx{z}$ then
      \begin{align}\label{eq:bounded-christoffel-convergence}
        \lim_{k \rightarrow \infty} \frac{1}{N} K_k(z) = \frac{\dx{\mu_{D}}}{\dx{V}} = \frac{v_D(z)}{w(z)},
      \end{align}
      with $K_k$ the $L^2_w$ reproducing kernel diagonal for $P_k$.
    \item Let $D \subset \R^d$ be an unbounded convex cone with $w = \exp(-2Q)$, and let $\rho(z) = \sqrt{w}(z)$ and $\dx{V}(z) = \dx{z}$. 
      Then 
      \begin{align}\label{eq:unbounded-christoffel-convergence}
        \lim_{k \rightarrow \infty} \frac{1}{N} \rho^{2 k}(z) K^{(k)}_k(z) = \dfdx{\mu_{D,Q}}{z} = v_{D,Q}(z),
      \end{align}
      with $K^{(k)}_k$ the $L^2_{w^k}$ reproducing kernel diagonal for $P_k$.
  \end{enumerate}
  
\end{corollary}

\subsection{Optimal measures}\label{sec:background-optimal-measures}
The discussion of this section is not directly related to the goal of this paper, but this brief diversion provides the following useful message: If one chooses to perform an unweighted Monte Carlo least-squares approximation (Algorithm \ref{alg:unweighted-mc}) with a polynomial subspace $P_k$, then sampling according to the equilibrium measure ($v = v_D$ for bounded $D$, $v = v_{D,Q}(k^{1/t}z )$ for unbounded $D$) gives the $k$-asymptotically optimal sampling criterion in the sense of Theorem \ref{thm:ls-stability}. This observation is essentially a corollary on a convergence result of ``optimal measures" as presented in \cite{bloom_convergence_2010}. \rev{The outline of this section is as follows:
  \begin{itemize}
    \item For any given probability measure\footnote{In this paper we are mainly interested in measures with Lebesgue densities $w$, but here we generalize by considering measures $\mu$ that are not necessarily absolutely continuous with respect to Lebesgue measure.} $\mu$, the standard Monte Carlo procedure (Algorithm \ref{alg:unweighted-mc}) has a required sample count criterion for stability and accuracy (Theorem \ref{thm:ls-stability}).
    \item One can ask ``for which $\mu$ is the procedure most stable?" and, because of \eqref{eq:cohen-sample-size}, this stability can be quantified by minimizing the maximum of the $\mu$-reproducing kernel diagonal $K$. The most stable measure will depend on the polynomial space $P$ considered, and it is natural to consider total degree spaces $P_k$.
    \item This leads to a definition of an ``optimal" measure $\mu_k$ corresponding to each degree, which is considered in \cite{bloom_convergence_2010}.
    \item Any sequence of optimal measures converges to the equilibrium measure. This motivates the message that, \textit{asymptotically} in $k$, the equilibrium measure is the most stable measure from which to perform discrete least-squares polynomial approximation.
  \end{itemize}
  We note that the above does not imply anything about approximation quality; in practice one is interested in approximation in a specific norm and is not terribly concerned about the defined optimal measures above. We only seek to point out that, asymptotically in $k$, the most stable recontruction procedure would result from reconstruction according to the equilibrium measure, and therefore is a heuristic motivation for the CLS sampling choice.
}


On a compact domain $D \subset \R^d$, consider approximation on the total degree space $P_k$. Let $\phi_\alpha(z;\mu)$ denote the orthonormal polynomials for $P_k$ under the $L^2$ norm with the probability measure $\mu$ on $D$. We are interested in choosing $\mu$ to optimize stability for a discrete least-squares problem. Thus, we consider the maximum value of the reproducing kernel diagonal as a function of the measure $\mu$:
\begin{align*}
  \kappa_k(\mu) &= \max_{z \in D} \sum_{|\alpha| \leq k} \phi^2_\alpha(z;\mu), &
  \int_D \phi_\alpha(z; \mu) \phi_\beta(z; \mu) \dx{\mu(z)} = \delta_{\alpha,\beta}
\end{align*}
If we ask for the $\mu$ that minimizes $\kappa_k(\mu)$ over all probability measures, this leads to the notion of \textit{optimal measures} as defined in \cite{bloom_convergence_2010}. Following the work in \cite{bloom_convergence_2010}, a measure $\mu_k$ is \textit{optimal} for $D$ and $P_k$ if, for all probability measures $\mu$ on $D$,
\begin{align*}
  \kappa_k(\mu_k) \leq \kappa_k(\mu).
\end{align*}
Algorithm \ref{alg:unweighted-mc} is most efficient for approximation with $P_k$ in the sense of Theorem \ref{thm:ls-stability} and condition \eqref{eq:cohen-sample-size} when the weight $w$ corresponds to the measure $\mu_k$, because this choice of measure produces the minimal value of $\kappa_k = \|K_k\|_\infty$.

\rev{The notion of an optimal measure changes slightly for the weighted, unbounded case: we essentially need to build the exponentially-decaying part of the weight (i.e., $\exp(-2 Q)$) into the least-squares problem in an intrinsic way, and then ask for the (optimal) measure that results in the most stable procedure for this exponentially-weighted least-squares problem.}
We let the exponentially-decaying part of the problem be defined by $\sqrt{w} = \rho = \exp(-Q)$ on an unbounded $D$. By building in the weights $\rho^{2 k}$ into the least-squares formulation, then the appropriate version of the quantity $\|K_k\|_\infty$ is given by $\sup_{z\in D} \rho^{2 k}(z) \sum_\alpha \left[\phi^{(k)}_{\alpha}\right]^2(z)$, where the $\phi^{(k)}$ are $\rho^{2 k}$-orthonormal as defined in \eqref{eq:phi-varying-weight}. Similar to the bounded case, we proceed to replace the measure $\rho^{2 k}(z) \dx{z}$ in \eqref{eq:phi-varying-weight} with $\rho^{2 k}(z) \dx{\mu(z)}$ for some probability measure $\mu$ on $D$, and define the resulting kernel maximum:
\begin{align*}
  \kappa_{k,\rho}(\mu) &= \max_{z \in D} \rho^{2 k}(z) \sum_{|\alpha| \leq k} \left[\phi^{(k)}_\alpha(z;\mu)\right]^2 
\end{align*}
The measure $\mu_k$ is an optimal measure for $D$ with weight $\rho$ if, for all probability measures $\mu$ on $D$:
\begin{align*}
  \kappa_{k,\rho}(\mu_k) \leq \kappa_{k,\rho}(\mu)
\end{align*}
Again this notion of an optimal measure indicates which sampling measure produces the smallest sampling size requirement in \eqref{eq:cohen-sample-size}.

In either the bounded or unbounded case we want to sample from $\mu_k$ to achieve an optimal stability factor. The following main result from \cite{bloom_convergence_2010} indicates that, as the polynomial degree $k$ tends to infinity, any sequence of (weighted) optimal measures converges to the (weighted) equilibrium measure.
\begin{theorem}[\cite{bloom_convergence_2010}]\label{thm:optimal-measures}
  Let $\mu_k$ be an optimal measure for $P_k$ on $D$ with weight $\rho = \exp(-Q)$. We have (i) for each $\mu_k$, $\kappa_{k,\rho} = N$ $\mu_{D,Q}\textrm{-}a.e.$, and (ii) $\lim_{k\rightarrow\infty} \mu_k = \mu_{D,Q}$ weakly.
\end{theorem}
Note that the above result holds also in the unweighted case $Q \equiv 0$. While computing $\mu_k$ for each $k$ will not be tractable in most situations, the result indicates that the optimal sampling measure for these least-squares problems must \textit{asymptotically} be $\mu_D$ in the bounded domain case, or a scaled version of $\mu_{D,Q}$ for the unbounded case. \rev{This result does not imply any optimality for a fixed $k$, and so in principle sampling with the equilibrium measure may be quite suboptimal if $k$ is small enough so that $\mu_k$ deviates significantly from $\mu_{D,Q}$.}

Note also that Theorem \ref{thm:optimal-measures} indicates that the stability factor $\kappa_k$ \textit{asymptotically} attains its optimal (minimal) value of $N$, which, according to Theorem \ref{thm:ls-stability} results in \textit{asymptotically} simple log-linear scaling of $S$ with respect to $N$, the best possible sample count criterion in the sense of Theorem \ref{thm:ls-stability}. \rev{Therefore, if $k$ could be taken very large, then sampling with the equilibrium measure will eventually produce a near-optimal $S \log S \gtrsim N$ sampling criterion for stability. While this seems promising since the asymptotic result is dimension-independent, it is not clear how large $k$ must be relative to $d$ to see this asymptotic behavior. In addition, it is computationally infeasible to use large $k$ for high-dimensional simulations since $\dim P_k \sim k^d$.}

%% file: content/proofs.tex
\section{Asymptotics of the Christoffel Least Squares algorithm}\label{sec:cls-theory}
This section concentrates on showing that the limiting behavior of Algorithms \ref{alg:cls-bounded} and \ref{alg:cls-unbounded} is stable and accurate. Our estimates depend on a discrepancy measure between the orthogonality weight $w$ and the effective CLS weight $\widetilde{w} \triangleq v \frac{N}{K}$, with $v$ the sampling density prescribed in Sections \ref{sec:cls-bounded} or \ref{sec:cls-unbounded}. Our convergence analysis is less constructive than the stability analysis, because the former depends on a $w$ versus $\widetilde{w}$ reprojection error, which is not easily computable.

\rev{We recall some of our notation in Table \ref{tab:notation} for clarity: $\Lambda$ is a multi-index set that defines a polynomial subpsace $P$, its $L^2_w$ reproducing kernel diagonal $K$, the $L^2_w$-orthogonal projector $\Pi$ whose range is $P$, the $L^2_{\widetilde{w}}$-orthogonal projector $\widetilde{\Pi}$ whose range is $P$, and the $L^2_{\widetilde{w}}$ Gramian matrix of the $\phi_n$, $\V{R}$. When $\Lambda = \Lambda_k$ corresponds to the polynomial space of total degree $k$, we use the abbreviated versions $P_k$, $\Pi_k$, $\widetilde{\Pi}_k$, and $\V{R}_k$.}

\subsection{CLS for bounded domains}\label{cls-proof-bounded}
We assume the pair $(D,w)$ are bounded-admissible in the sense of Section \ref{sec:setup}. Our results in this section and the one immediately following are essentially adaptations of the results reproduced in Sections \ref{sec:background-ls} and \ref{sec:background-equilibrium}. This section deals with compact domains.

The CLS framework is a weighted least-squares formulation; alternatively, we may consider it an unweighted least-squares problem with the non-polynomial functions
\begin{align}\label{eq:psi-def}
  \psi_\alpha(z) = \frac{\sqrt{N} \phi_\alpha(z)}{\sqrt{\sum_{\alpha \in \Lambda} \phi_\alpha^2(z)}},
\end{align}
followed by sampling with the equilibrium measure weight $v(z)\dx{z} = \dx{\mu_D}(z)$. This essentially uses a modified weight function that approximates $w$:
\begin{align*}
  \widetilde{w}(z) = \frac{N}{K(z)} v_D(z)
\end{align*}
The functions $\psi_\alpha$ are not exactly orthogonal with respect to the weight $v_D(z)$, and we will need a quantification of this non-orthogonality behavior as a function of the index set $\Lambda$. For a given $w$ and fixed index set $\Lambda$, define
\begin{align}\label{eq:qn-def}
  (R)_{\alpha,\beta} &= \int_D \psi_\alpha(z) \psi_\beta(z) \dx{\mu_D} = \int_{D} \phi_\alpha(z) \phi_\beta(z) \widetilde{w}(z) \dx{z},
\end{align}
so that the $N \times N$ matrix $\V{R}$ is defined. 
We emphasize that the functions $\psi_\alpha$ and the weight $\widetilde{w}$ depend on $\Lambda$. We use the notation $\V{R} = \V{R}_k$ to denote the special case of $\Lambda = \Lambda_k$. Owing to asymptotics of the Christoffel function, any fixed $(\alpha,\beta)$ entry of the matrix $\V{R}_k$ converges to the corresponding entry of the identity matrix.
\begin{proposition}
  For any fixed $\alpha, \beta$, the quantity in \eqref{eq:qn-def} for the polynomial space $P_k$ satisfies
  \begin{align*}
    \lim_{k \rightarrow \infty} (R_k)_{\alpha,\beta} = \delta_{\alpha,\beta}
  \end{align*}
\end{proposition}
\begin{proof}
  The entries of $\V{R}_k$ are given by 
  \begin{align*}
    (R_k)_{\alpha,\beta} = \int_D \phi_\alpha(z) \phi_\beta(z) \frac{N}{K_k(z)} v(z) \dx{z}.
  \end{align*}
  The result \eqref{eq:bounded-christoffel-convergence} implies that $\widetilde{w} = \frac{N}{K_k(z)} v(z)$ converges to $w(z)$ weakly, so that for fixed $\alpha$, $\beta$,
  \begin{align*}
    (R_k)_{\alpha,\beta} \rightarrow \int_D \phi_\alpha \phi_\beta w\, \dx{z} = \delta_{\alpha,\beta}.
  \end{align*}
\end{proof}
The above is an asymptotic result indicating that individual terms of the matrix $\V{R}_k$ behave like terms of the identity. However, this cannot be used to conclude that $\V{R}_k$ is close to the identity matrix in, e.g., the induced $\ell^2$ norm for increasing $k$ since the size of $\V{R}_k$ also increases with $k$. To illustrate this, we compile results for the one-dimensional domains $D = [-1,1]$ with symmetric Jacobi weights $w(z) = (1-z^2)^\alpha$ in the left-hand pane of Figure \ref{fig:qn-error}. These results alone cannot even be used to conclude that $d=1$ cases for $\V{R}_k$ are well-conditioned. However, one can combine Figure \ref{fig:qn-error} with Figure \ref{fig:qn-mineig} to see that in fact the $\V{R}_k$ are relatively well-behaved.
\begin{figure}
\begin{center}
  \includegraphics[width=\textwidth]{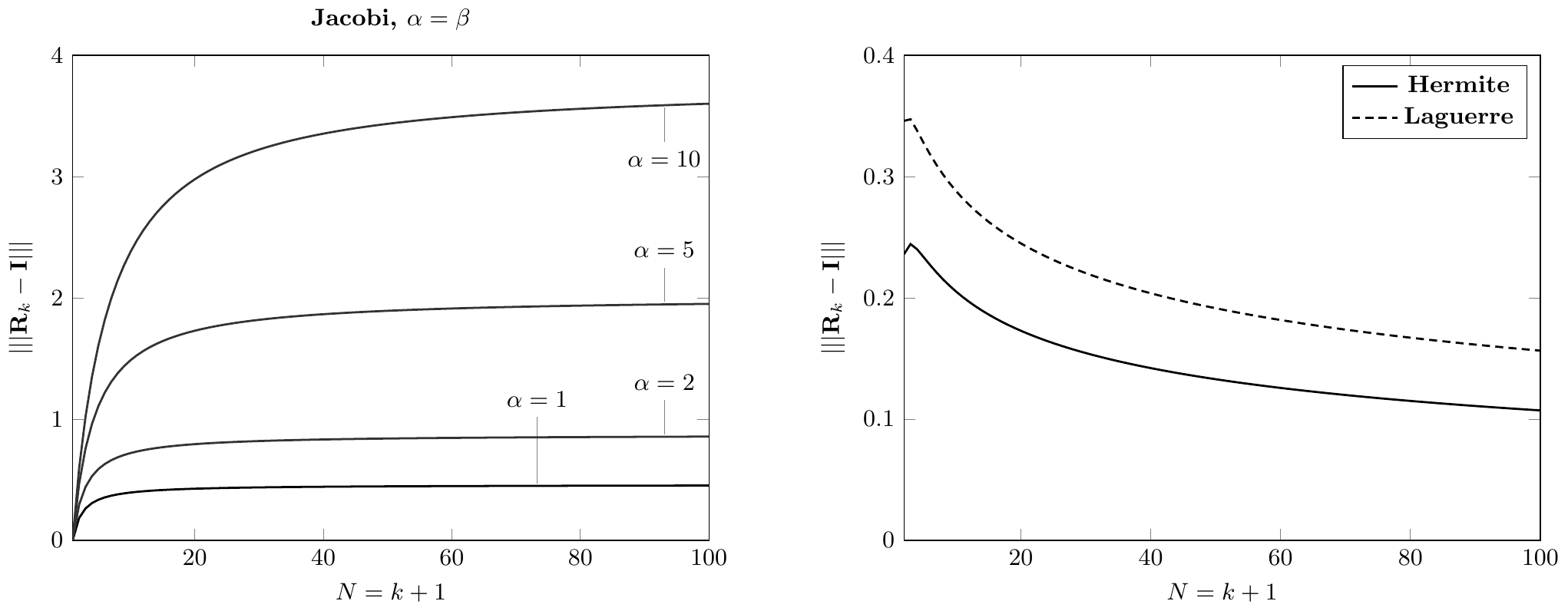}
\end{center}
\caption{Deviation of the matrix $\V{R}_k$ from the $N \times N$ identity for one-dimensional domains $D$. Left: bounded domain $D = [-1,1]$ with $\V{R}_k$ defined in \eqref{eq:qn-def} with one-dimensional orthogonality density $w(z) = (1-z^2)^\alpha$. Right: unbounded domain $D$ with weights $w(z) = \exp(-z^2)$ and $w(z) = \exp(-z)$ defining $\V{R}_k$ as in \eqref{eq:qn-def-unbounded}.}\label{fig:qn-error}
\end{figure}

Turning to stability, since $\V{R}$ is not the identity, we can only expect the CLS normal equations matrix $\V{G}$ to converge to $\V{R}$ as the sample count increases. (Recall the definition of $\V{G}$ from \eqref{eq:normal-equations}.) The same arguments as in \cite{cohen_stability_2013} may be applied to conclude an analogue of the stability result in Theorem \ref{thm:ls-stability}: A sampling size criterion implies that the discrete Gramian $\V{G}$ of the CLS procedure is close to $\V{R}$ with high probability.
\begin{theorem}\label{thm:cls-bounded-stability}
  For a compact domain $D$ and admissible weight $w$ with index set $\Lambda$, consider the CLS algorithm, Algorithm \ref{alg:cls-bounded}. If the number of samples $S$ satisfies
  \begin{align}\label{eq:christoffel-sample-requirement}
    \frac{S}{N \log S} \geq \left[\frac{1 + r}{c_\delta}\right] \frac{1}{\lambda_{\textrm{min}}(\V{R})},
  \end{align}
  for $\delta \in (0,1)$ and $r >0$, with $c_\delta$ defined in \eqref{eq:cohen-sample-size}, then 
  \begin{align*}
    \mathrm{Pr} \left[ \frac{\vertiii{\V{G} - \V{R}}}{\vertiii{\V{R}}} > \delta \right] \leq \frac{2}{S^r}
  \end{align*}
\end{theorem}
\begin{proof}
  The proof of Theorem \ref{thm:ls-stability} in \cite{cohen_stability_2013} is easily amended for our purposes. Since $\V{R}$ is symmetric positive-definite, its symmetric positive-definite square root $\V{R}^{1/2}$ is well-defined. Then 
  \begin{align*}
    \vertiii{\V{G} - \V{R}} &= \vertiii{ \V{R}^{1/2} \left( \V{R}^{-1/2} \V{G} \V{R}^{-1/2} - \V{I} \right) \V{R}^{1/2} } \\
                              &\leq \vertiii{ \V{R} } \vertiii{ \V{R}^{-1/2} \V{G} \V{R}^{-1/2} - \V{I}}
  \end{align*}
  And so
  \begin{align}\label{eq:thm-temp3}
    \mathrm{Pr} \left[ \frac{\vertiii{\V{G} - \V{R}}}{\vertiii{\V{R}}} > \delta \right] \leq \mathrm{Pr} \left[ \vertiii{ \V{R}^{-1/2} \V{G} \V{R}^{-1/2} - \V{I}} > \delta \right]
  \end{align}
  The CLS Gramian matrix $\V{G}$ can be decomposed into a sum of independent matrices
  \begin{align*}
    \V{G} &= \sum_{s=1}^S \V{Y}_s, & \left(\V{Y}_s \right)_{\alpha,\beta} &= \frac{N}{S} \left[ \frac{ \phi_\alpha(z_s) \phi_\beta(z_s)}{ K(z_s)} \right].
  \end{align*}
  Defining $\V{X}_s = \V{R}^{-1/2} \V{Y}_s \V{R}^{-1/2}$, then $\V{R}^{-1/2} \V{G} \V{R}^{-1/2} = \sum_s \V{X}_s$. The spectral norm of $\V{X}_s$ satisfies
  \begin{align*}
    \vertiii{\V{X}_s} \leq \vertiii{\V{R}^{-1}} \vertiii{\V{Y}_s}
  \end{align*}
  Since $\V{Y}_s$ is a rank-1 matrix formed from the outer product of $\left( \sqrt{\frac{N}{S K(z_s)}} \phi_\alpha(z_s) \right)_\alpha$ with itself, then 
  \begin{align*}
    \vertiii{\V{X}_s} \leq \frac{N}{S \lambda_\mathrm{min}(\V{R})} \frac{ \sum_{\alpha \in \Lambda} \phi^2_\alpha(z_s)}{ K(z_s) } = \frac{N}{S \lambda_\mathrm{min}(\V{R})}
  \end{align*}
  with probability 1. The summed expected value of $\V{X}_s$ yields the identity matrix: $\sum_{s=1}^S \E \V{X}_s = \V{R}^{-1/2} \V{R} \V{R}^{-1/2} = \V{I}$. 
  
  Now we can use the matrix Chernoff bound from \cite{tropp_user-friendly_2012}: for any collection of independent random matrices $\V{X}_s$ satisfying $\vertiii{\V{X}_s} \leq M$, then 
  \begin{align*}
    \mathrm{Pr} \left[ \lambda_{\textrm{min}} \left( \sum_{s=1}^S \V{X}_s \right) \leq (1 - \delta) \lambda_{\textrm{min}} \left(\sum_{s=1}^S \E \V{X}_s \right)\right] &\leq N \exp\left(-\frac{c_\delta \lambda_{\textrm{min}} \left( \sum_{s=1}^S \E \V{X}_s \right)}{M}\right), \\
    \mathrm{Pr} \left[ \lambda_{\textrm{max}} \left( \sum_{s=1}^S \V{X}_s \right) \geq (1 + \delta) \lambda_{\textrm{max}} \left(\sum_{s=1}^S \E \V{X}_s \right)\right] &\leq N \exp\left(-\frac{c_\delta \lambda_{\textrm{max}} \left( \sum_{s=1}^S \E \V{X}_s \right)}{M}\right),
  \end{align*}
  where $c_\delta = \delta + (1-\delta) \log(1 - \delta) \in (0,1)$. We use $\V{G} = \sum_s \V{X}_s$ and $\sum_s \E \V{X}_s = \V{I}$ along with the bound $M = \frac{N}{S \lambda_{\mathrm{min}}(\V{R})}$. Thus we have 
  \begin{align}\nonumber
    \mathrm{Pr} \left[ \vertiii{ \V{R}^{-1/2} \V{G} \V{R}^{-1/2} - \V{I}} \geq \delta \right] &\leq N \left[ \exp \left( - \frac{c_{\delta} S \lambda_{\textrm{min}}(\V{R})}{N} \right) + 
    \exp \left( - \frac{c_{\delta} S \lambda_{\textrm{min}}(\V{R})}{N} \right)\right] \\\label{eq:thm-temp1}
    &\leq 2 S \exp \left( - \frac{c_{\delta} S \lambda_{\textrm{min}}(\V{R})}{N} \right)
  \end{align}
  If we require \eqref{eq:christoffel-sample-requirement}, then
  \begin{align}\label{eq:thm-temp2}
    \exp \left( - \frac{c_{\delta} S \lambda_{\textrm{min}}(\V{R})}{N} \right) \leq \exp\left( - (1 + r) \log S\right) = \frac{1}{S^{1+r}}.
  \end{align}
  Combining \eqref{eq:thm-temp3}, \eqref{eq:thm-temp1}, and \eqref{eq:thm-temp2} yields the result.
\end{proof}
\begin{remark}
  The above result is a stability estimate for a general weighted least-squares approach for any biased weight $\widetilde{w} \neq w$.
\end{remark}

We emphasize that we have only established stability of the least-squares problem relative to $\V{R}$. The required sample count for stability no longer depends on the normalized polynomial reproducing kernel $K(z)/N$, but instead on $\lambda_{\textrm{min}}^{-1}(\V{R})$, which is a stability measure for $\V{R}$. (Compare \eqref{eq:cohen-sample-size} with \eqref{eq:christoffel-sample-requirement}.) In Figure \ref{fig:qn-mineig} we plot the inverse of the minimum eigenvalue for one-dimensional cases: symmetric Jacobi polynomials (with parameter $\beta > -1$) and also for Hermite and Laguerre polynomials. 

For most of these one-dimensional cases of interest with the space $P_k$, the factor $1/\lambda_{\textrm{min}}(\V{R}_k)$ is less than 2. This is in stark contrast to the results in Figure \ref{fig:maxK-plots} where the stability factor of $\frac{\|K\|_\infty}{N}$ of \eqref{eq:cohen-sample-size} from similar one-dimensional scenarios is extremely large.

While these results are promising in one dimension for a large, fixed degree $k$, they will likely deteriorate as the dimension $d$ is increased with a fixed $k$. Similarly, Figures \ref{fig:qn-mineig} and \ref{fig:qn-error} can be used to conclude that $\V{R}_k$ is quite well-conditioned for classical one-dimensional problems, but this is unlikely to persist for large dimensions.

\begin{figure}
\begin{center}
  \includegraphics[width=\textwidth]{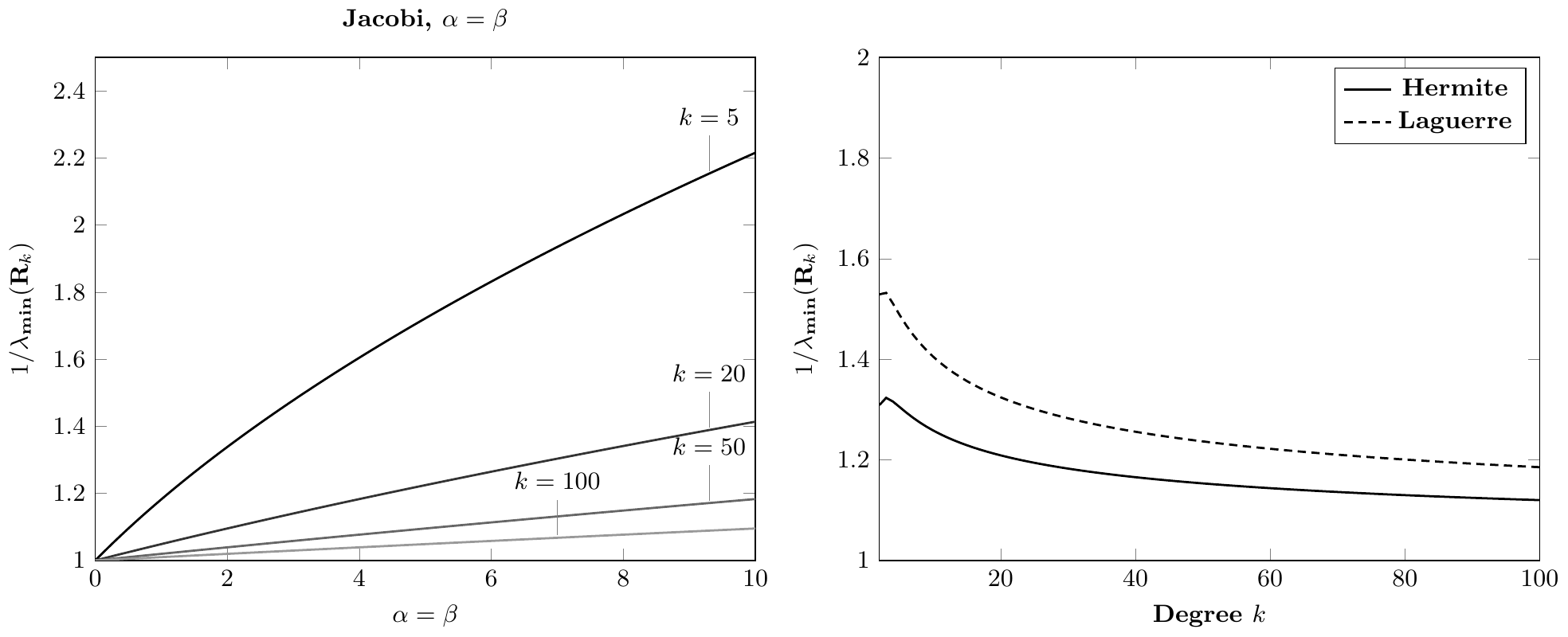}
\end{center}
\caption{Inverse of the minimum eigenvalue of the matrix $\V{R}_k$ for various one-dimensional polynomial families. Left: Jacobi polynomials with symmetric parameters $\alpha = \beta$. Right: Hermite and Laguerre polynomials.}\label{fig:qn-mineig}
\end{figure}

We are able to prove a convergence result if $|f(z)| \leq L$ for all $z \in D$. The random function $\widetilde{\Pi}^S f$ is the CLS-bounded discrete projection. We introduce a truncation of this discrete projection:
\begin{align}\label{eq:cls-estimator}
  \widetilde{f}(z) = T_L \left[\widetilde{\Pi}^S f\right],
\end{align}
where $T_L(x) = \mathrm{sgn} (x) \min\left\{|x|, L\right\}$. 
\rev{We need an additional discrepancy measure between $w$ and $\widetilde{w}$ that depends on the function $f$ being approximated. We define this as the $L^2_w$ error between $\widetilde{\Pi} f$ and $\Pi f$:
  \begin{align}\label{eq:data-discrepancy}
    d(f) \triangleq d\left(w,\widetilde{w},P;\,f\right) = \left\| \widetilde{\Pi} f - \Pi f\right\|_w = \left\|\widetilde{\Pi} \left(I - \Pi \right) f\right\|_w
  \end{align}
  Note that in all the follows we avoid explicit indication that $d$ depends on $w$, $\widetilde{w}$, and the polynomial space $P$ to limit notational clutter. Also note that, when $f \not\in P$, then the size of $d(f)$, relative to the optimal error $\|f - \Pi f\|_w$, is bounded by the operator norm of $\widetilde{\Pi}$ on the kernel of $\Pi$ in $L^2_w$:
  \begin{align*}
    d(f) \leq \|(I - \Pi)f\|_w \sup_{f \in \ker \Pi} \frac{\left\|\widetilde{\Pi}\left(I - \Pi\right) f\right\|_w}{\|(I - \Pi) f\|_w} = \left\|\widetilde{\Pi}\right\|_{P^{\perp}} \|(I - \Pi) f\|_w 
  \end{align*} 
}
Above, $P^{\perp}$ is the orthogonal complement of $P$ in $L_w^2$. It is unclear whether or not the above operator norm of $\widetilde{\Pi}$ can be computed or estimated in general situations. 

Following the arguments in \cite{cohen_stability_2013}, we can bound the error for the truncated CLS estimator.
\begin{theorem}\label{thm:bounded-convergence}
  For a given function $f$, let $\Pi f$ be the $L^2_w$ projection onto a polynomial space $P$. If the number of samples $S$ in the CLS-bounded algorithm satisfies \eqref{eq:christoffel-sample-requirement}, then the mean-square $L^2_w$ error of the truncated CLS approximation $\widetilde{f}$ defined in \eqref{eq:cls-estimator} satisfies
  \begin{subequations}
    \rev{ 
  \begin{align}\label{eq:bounded-convergence}
      \E \left[ \|f - \widetilde{f}\|_w^2 \right] &\leq \left\|f - \Pi f\right\|_w^2 + \frac{\varepsilon(S)}{\lambda_{\mathrm{min}}(\V{R})} \left\|f - \Pi f\right\|^2_{\widetilde{w}} + \frac{8 L^2}{S^r} + 4\kappa^2(\V{R}) d^2(f) 
  \end{align}
  }
  \end{subequations}
  with $\varepsilon(S) \triangleq \frac{2 - 2 \log 2}{(1+r) \log S} \rightarrow 0$ as $S \rightarrow \infty$, and $\kappa(\V{R}) = \lambda_{\mathrm{max}}(\V{R}) / \lambda_{\mathrm{min}}(\V{R})$ the 2-norm condition number of $\V{R}$.
\end{theorem}
\begin{proof}
  Our proof follows that of Theorem 2 in \cite{cohen_stability_2013}. Under the sampling condition \eqref{eq:christoffel-sample-requirement} with $\delta = \frac{1}{2}$, we have the following inequality with probability at least $1 - \frac{2}{S^r}$:
\begin{align}\label{eq:Ginv-bound}
  \lambda_{\mathrm{max}}\left(\V{G}^{-1}\right) \leq \lambda_{\max} \left(\V{R}^{-1/2}\right) \lambda_{\mathrm{max}} \left( \V{R}^{1/2} \V{G}^{-1} \V{R}^{1/2}\right) \lambda_{\mathrm{max}} \left(\V{R}^{-1/2}\right) \leq \frac{2}{\lambda_{\mathrm{min}} \left(\V{R}\right)}
\end{align}
We denote the probabilistic set under which this happens as $\Omega_+$, and $\Omega_-$ the set under which this fails. Then
\begin{align*}
  \E \left[ \|f - \widetilde{f}\|_w^2 \right] \leq \E \left[ \| f - \widetilde{\Pi}^S f \|_w^2\, \big| \, \Omega_+\right] + \frac{8 L^2}{S^r}
\end{align*}
where we have used the fact that $\|T_L[f]\|_w \leq \|f\|_w$, and $T_L[f] = f$ if $|f| \leq L$. Let $g = f - \Pi f$. We note that since $g$ is $L^2_w$-orthogonal to $\Pi f$, and $\widetilde{\Pi}^S$ is the identity on $P$, then
\begin{align*}
  f - \widetilde{\Pi}^S f = g - \widetilde{\Pi}^S g,
\end{align*}
so that
\begin{align*}
  \left\|g - \widetilde{\Pi}^S g \right\|_w^2 = \left\|g\right\|_w^2 + \left\| \widetilde{\Pi}^S g\right\|_w^2 = \left\|g\right\|^2_w + \|\V{b}\|^2,
\end{align*}
where the coefficients $b_j$ in the vector $\V{b}$ are the polynomial coefficients recovered from the CLS approximation on the function $g$. Thus, we have
\begin{subequations}
\begin{align}\label{eq:thm-1-1}
  \E \left[ \|f - \widetilde{f}\|_w^2 \right] \leq \|g\|_w^2 + \frac{8 L^2}{S^r} + \E \left[ \|\V{b}\|^2 \, \big|\, \Omega_+ \right]
\end{align}
We have \eqref{eq:Ginv-bound}, and using the normal equations \eqref{eq:normal-equations} on the event $\Omega_+$ \rev{(in \eqref{eq:normal-equations}, replace $\V{f} \gets \V{g}$ and $\V{c} \gets \V{b}$)}, we have
\begin{align}\label{eq:thm-1-2}
  \|\V{b}\|^2 \leq \frac{4}{S^2 \lambda_{\mathrm{min}}^2(\V{R})} \| \V{V}^T \V{K} \V{g}\|^2 = \frac{4}{\lambda_{\mathrm{min}}^2(\V{R})} \sum_{n=1}^N \left\langle \phi_n(z), \frac{N}{K(z)} g(z) \right\rangle_S^2
\end{align}
where $\langle \cdot, \cdot \rangle_S$ is defined in \eqref{eq:discrete-inner-product}.
Letting $z_s$ denote random iid variables distributed according to the sampling density $v$, then each summand on the right-hand side above has expected value given by
\begin{align*}
  \E \left( \frac{1}{S} \sum_{s=1}^S \frac{N}{K(z_s)} \phi_n(z_s) g(z_s)\right)^2 &= \frac{1}{S^2} \sum_{s,r=1}^S \E \frac{N^2}{K(z_s) K(z_r)} \phi_n(z_s) \phi_n(z_r) g(z_s) g(z_r) \\ 
                                                                                      &= \frac{S}{S^2} \E \frac{N^2}{K^2(z_1)} \phi^2_n(z_1) g^2(z_1) + \frac{S(S-1)}{S^2} \left[ \E \frac{N}{K(z_1)} \phi_n(z_1) g(z_1)\right]^2
\end{align*}
Summing over $n$, we have
\begin{align*}
  \E \|\V{V}^T \V{K} \V{g} \|^2 &= \frac{1}{S} \E \frac{N^2}{K^2(z_1)} K(z_1) g^2(z_1) + \frac{S-1}{S} \sum_{n=1}^N \left[\E \frac{N}{K(z_1)} \phi_n(z_1) g(z_1)\right]^2 \\
                                &= \underbrace{\frac{N}{S}}_{\mathrm{(a)}} \underbrace{\E \frac{N}{K(z_1)} g^2(z_1)}_{\mathrm{(b)}} + \frac{S-1}{S} \underbrace{\sum_{n=1}^N \left[ \int_D \phi_n(z) g(z) \widetilde{w} \dx{z}\right]^2}_{\mathrm{(c)}}
\end{align*}
Term (a) on the right-hand side can be bounded by using the condition \eqref{eq:christoffel-sample-requirement}, so that $\mathrm{(a)} = \frac{N}{S} \leq \frac{\lambda_{\mathrm{min}}(\V{R}) c_\delta}{(1 + r) \log S}$. Term (b) is equal to $\|g\|^2_{\widetilde{w}}$. To bound (c), consider the continuous projection $\widetilde{\Pi} g = \sum_{n=1}^N d_n \phi_n$, with the $d_n$ solving the expected value of the normal equations \eqref{eq:normal-equations}:
\begin{align*}
  \V{R} \V{d} &= \V{h}, & (h)_n &= \int_D g(z) \phi_n(z) \widetilde{w}(z) \dx{z}
\end{align*}
We have
\begin{align*}
  \mathrm{(c)} = \sum_{n=1}^N \left[ \int_D \phi_n(z) g(z) \widetilde{w} \dx{z}\right]^2 = \|\V{h}\|^2 \leq \vertiii{\V{R}}^2 \, \|\V{d}\|^2 \leq \lambda_{\mathrm{max}}^2(\V{R}) \|\widetilde{\Pi} g\|^2_w
\end{align*}
Thus we have
\begin{align}\label{eq:thm-1-3}
  \E \|\V{V}^T \V{K} \V{g} \|^2 \leq \lambda_{\mathrm{min}}(\V{R}) \frac{c_\delta}{(1+r) \log \delta} \|g\|^2_{\widetilde{w}} + \lambda_{\mathrm{\max}}^2(\V{R}) \|\widetilde{\Pi} g\|^2_w
\end{align}
\end{subequations}
Combining \eqref{eq:thm-1-1}, \eqref{eq:thm-1-2}, and \eqref{eq:thm-1-3}, we have 
\begin{align*}
  \E \left[ \|f - \widetilde{f}\|_w^2 \right] \leq \|g\|_w^2 + \frac{4 c_\delta}{(1+r) \lambda_{\mathrm{min}}(\V{R}) \log S} \|g\|_{\widetilde{w}}^2 + \frac{8 L^2}{S^r} + 4 \kappa^2(\V{R}) \|\widetilde{\Pi} g\|_w^2,
\end{align*}
\rev{
  where the appropriate value of $c_\delta$ from \eqref{eq:cohen-sample-size} for $\delta = \frac{1}{2}$ should be used. Noting that the term $\left\| \widetilde{\Pi} g\right\|_w^2$ is precisely $d^2(f)$ proves the result.
}
\end{proof}
The above result is suboptimal in the sense that as $S \rightarrow \infty$, the error converges to a value that deviates from the optimal value of $\|g\|^2_w$. The suboptimal term involving $d^2(f)$ is a term that stems from the discrepancy between $w$ and $\widetilde{w}$. Indeed, $d(f) = \widetilde{\Pi} (I - \Pi) f$, which vanishes when $w$ and $\widetilde{w}$ coincide (or when $f$ is a polynomial in $P$ regardless of $\widetilde{w}$). A second apparently deviation of the CLS result above from the standard Monte Carlo estimate \eqref{eq:cohen-convergence} is that the second term on the right-hand side of \eqref{eq:bounded-convergence} is proportional to $\|f - \Pi f\|_{\widetilde{w}}$ rather than $\|f - \Pi f\|_w$. 

\rev{
\subsubsection{Size of the CLS discrepancy terms}\label{sec:cls-accuracy}
In this section we give empirical evidence to suggest that the standard Monte Carlo error estimate form Theorem \ref{thm:ls-stability} and the CLS error estimate derived in Theorem \ref{thm:bounded-convergence} give comparable bounds. The different error terms in the standard Monte Carlo convergence result \eqref{eq:cohen-convergence} and the CLS convergence result \eqref{eq:bounded-convergence} are, respectively,
\begin{align*}
  \varepsilon(S) \|g\|_w^2 \hskip 10pt \textrm{versus} \hskip 10pt \frac{\varepsilon(S)}{\lambda_{\textrm{min}}(\V{R})} \|g\|_{\widetilde{w}}^2 + 4 \kappa^2(\V{R}) d^2(f).
\end{align*}
Above, $g = f - \Pi f$ is the truncation error between $f$ and its $L^2_w$ projection onto the polynomial space $P$. To investigate how suboptimal the above terms from CLS algorithm are, we separate dependence on the sample count $S$ by assuming that $\varepsilon(S) = 1$. In this case, a measure of suboptimality of the CLS theory is given by the ratio of the above terms:
\begin{align}\label{eq:Delta-def}
  \Delta(f) \triangleq \frac{1}{\lambda_{\textrm{min}}(\V{R})} \frac{\|g\|_{\widetilde{w}}^2}{\|g\|^2_w} + 4 \kappa^2(\V{R}) \frac{d^2(f)}{\|g\|^2_w},
\end{align}
with $\Delta \approx 1$ indicating that the CLS and standard Monte Carlo error terms are of roughly the same magnitude. Values satisfying $\Delta \gg 1$ indicate that the CLS convergence terms are quite suboptimal. We consider the case $D = [-1,1]$ with the weight $w = 1$, so that the $\phi_n$ basis elements are orthonormal Legendre polynomials. We choose four test functions $f^{(q)}$ for $q = 0, 1, 2, 3$, where $q$ is an indicator of the smoothness of these functions:
\begin{align}\nonumber
  f^{(0)}(z) &= \left\{ \begin{array}{rl}
                        1, & z \leq \frac{1}{2} \\
                        -1, & z > \frac{1}{2}
                     \end{array}
                     \right. \\\label{eq:fq}
  f^{(q+1)}(z) &= \int_{-1}^z f^{(q)}(x) \dx{x}
\end{align}
These functions have $q$ derivatives in $L^2_w$ and have $\mathcal{O}(1)$ values on $[-1,1]$. In Figure \ref{fig:cls-accuracy-discrepancy} we show values of $\Delta$ for these four test functions, and see that $\Delta \approx 1$ in most scenarios, with $\Delta < 2$ in all cases tested. This suggest that the CLS error bound provided by \eqref{eq:bounded-convergence} is, in practice, as sharp as that provided by \eqref{eq:cohen-convergence}, at least for our choice of $w$ and $\widetilde{w}$. Similar tests in $d=1$ dimension yield similar results. 
\begin{figure}
\begin{center}
  \includegraphics[width=\textwidth]{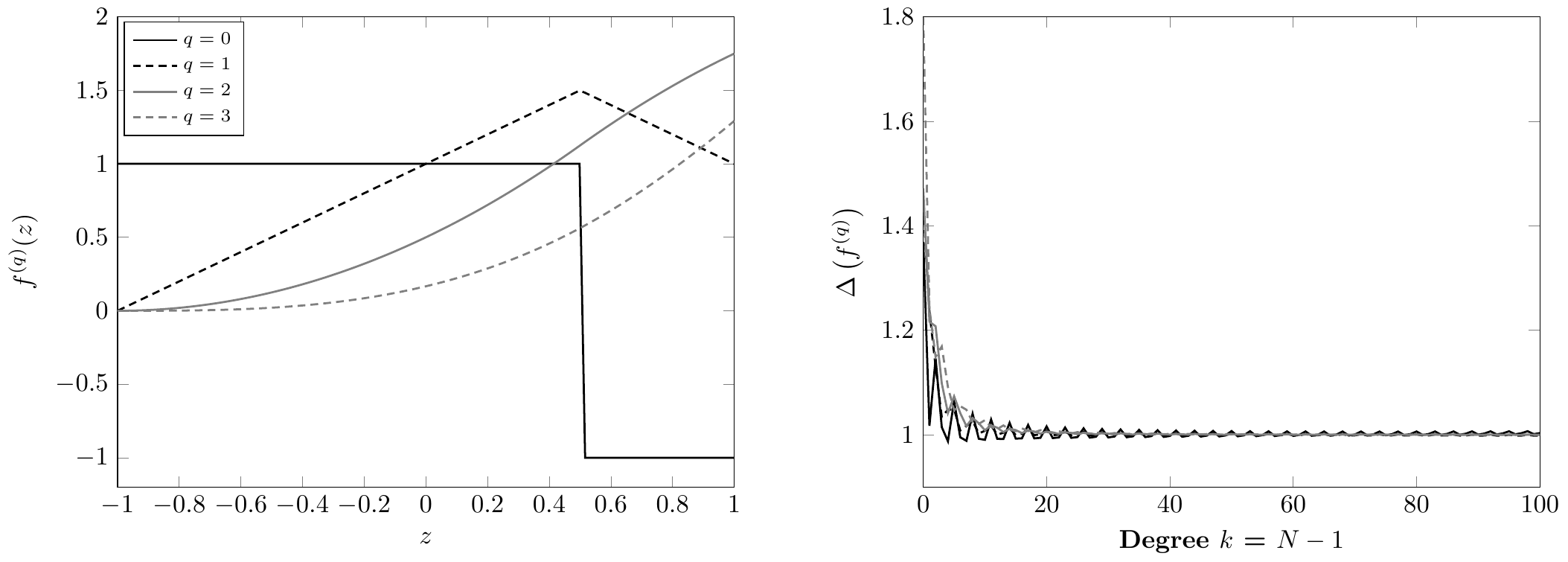}
\end{center}
\caption{\rev{Left: 4 test functions defined by \eqref{eq:fq} whose smoothness increases with $q$. Right: Size of the theoretical CLS bound \eqref{eq:bounded-convergence} relative to the size of the standard Monte Carlo bound \eqref{eq:cohen-convergence}, measured by the parameter $\Delta$ defined in \eqref{eq:Delta-def}. Values $\Delta \approx 1$ indicate that the CLS bound is comparable to the standard Monte Carlo bound.}}\label{fig:cls-accuracy-discrepancy}
\end{figure}
Thus, while the CLS error bound in Theorem \ref{thm:bounded-convergence} initially appears weaker than the standard Monte Carlo estimate in \ref{thm:ls-stability}, our limited testing indicates that they are comparable. We close this section by noting that the first term in \eqref{eq:Delta-def} involving $\|g\|_{\widetilde{w}}/\|g\|_w$ is the major contributer to $\Delta$, having value approximately 1 in our tests. This is consistent with expectations since $\widetilde{w} \approx w$. Therefore, the second term involving $d^2(f)$ appeared to have comparatively little contribution to the value of $\Delta$.
}

\subsection{Unbounded domains}\label{sec:cls-proof-unbounded}

We assume the pair $(D,w)$ is unbounded-admissible in the sense of Section \ref{sec:setup}. Our results are essentially identical to the bounded domain case, but some of the definitions change. The CLS method in this case is given by Algorithm \ref{alg:cls-unbounded}, with the sampling density $v$ given by \eqref{eq:cls-unbounded-weight}.

The Christoffel-weighted functions $\psi_\alpha$ are identical to the bounded case given by \eqref{eq:psi-def}. The surrogate weight $\widetilde{w}$ is defined as 
\begin{align*}
  \widetilde{w}(z) = v(z) \frac{N}{K(z)} = k^{-d/t} \frac{N}{K(z)} v_{D,Q}\left(k^{-1/t} z\right)
\end{align*}
where $k = \max_{\alpha \in \Lambda} |\alpha|$. We let $A \triangleq \mathrm{supp}\, \mu_{D,Q}$ be the support of the equilibium measure, and the analogous definition of \eqref{eq:qn-def} in the unbounded case is 
\begin{align}\label{eq:qn-def-unbounded}
  (R)_{\alpha,\beta} &= k^{-d/t} \int_{ k^{1/t} A} \psi_\alpha(z) \psi_\beta(z) \dx{\mu_{D,Q}(k^{-1/t} z)} = \int_{k^{1/t} A} \phi_\alpha(z) \phi_\beta(z) \widetilde{w}(z) \dx{z}.
\end{align}

Before continuing, we need a result that relates polynomials orthonormal under $w^k$ to those orthonormal under $w$, assuming $w = \exp(-2 Q)$ with $Q$ a homogeneous function.
\begin{lemma}\label{lemma:orthogonality-expanding}
  Let a weight function $w = \exp(-2 Q)$ with $Q$ satisfying \eqref{eq:q-homogeneous} be given with homoegeneity exponent $t$ on an unbounded conic domain $D$. Let $\phi_\alpha$ be a polynomial family that is $L^2$ orthonormal under weight $w$. Then a family of polynomials $\phi^{(k)}_\alpha$ that is orthonormal under $w^{k}$ is 
  \begin{align*}
    \phi^{(k)}_\alpha(z) = k^{d/2t} \phi_\alpha\left(k^{1/t} z\right).
  \end{align*}
\end{lemma}
\begin{proof}
  Since $\phi_\alpha$ is orthonormal under $w$, then
  \begin{align*}
    \int_{D} \phi_\alpha(z) \phi_\beta(z) w(z) \dx{z} = \delta_{\alpha,\beta}.
  \end{align*}
  By assumption \eqref{eq:q-homogeneous}, $w^k(z) = w\left( k^{1/t} z\right)$. Then making the substitution $z \gets k^{1/t} z$ in the relation above yields
  \begin{align*}
    k^{d/t} \int_{D} \phi_\alpha(k^{1/t} z) \phi_\beta(k^{1/t}) w^k(z) \dx{z} = \delta_{\alpha,\beta}
  \end{align*}
  Since $\phi_\alpha\left(C z\right)$ is still a polynomial of degree $\alpha$ for any constant $C$, this proves the result.
\end{proof}

Again, the matrix $\V{R}_k$ corresponding to the total-degree polynomial space $P_k$ is reasonably well-behaved with respect to the identity, as can be seen from the right-hand pane of Figure \ref{fig:qn-error}, and owing to the result \eqref{eq:unbounded-christoffel-convergence} from Corollary \ref{cor:christoffel-convergence}, individual entries of $\V{R}_k$ converge to the Kronecker delta.
\begin{proposition}
  Fix multi-indices $\alpha$ and $\beta$. 
  Then the entries of $\V{R}_k$ in \eqref{eq:qn-def-unbounded} obey
\begin{align*}
    \lim_{\revision{k}\rightarrow \infty} (R_k)_{\alpha,\beta} = \delta_{\alpha,\beta}
\end{align*}
\end{proposition}
\begin{proof}
  Since $\rho = \sqrt{w}$ is negative-log-homogeneous of degree $t$, then $\rho^{n}(z) = \rho(z n^{1/t})$, or in other words, $w^k\left(k^{-1/t} z\right) = w(z)$.
  Let $\phi_\alpha^{(k)}$ be the polynomial family orthogonal under $w^k$. By Lemma \ref{lemma:orthogonality-expanding}, we have
  \begin{align*}
    k^{-d/t} {K}^{(k)}_k\left(k^{-1/t}z\right) = k^{-d/t} \sum_{|\alpha| \leq k} \left[\phi^{(k)}_\alpha\left(k^{-1/t} z\right)\right]^2 = k^{-d/t} \sum_{|\alpha| \leq k} \left[ k^{d/2t} \phi_\alpha\left(z\right) \right]^2 = K_k(z)
  \end{align*}
  With $\rho^{2k}(z) = w^k(z)$, we can use \eqref{eq:unbounded-christoffel-convergence} to conclude: 
  \begin{align*}
    \frac{N v_{D,Q}( k^{-1/t} z)}{{K}^{(k)}_k\left(k^{-1/t} z\right) w^k \left( k^{-1/t} z\right)}  \rightarrow 1
  \end{align*}
  weakly on compact sets.
  Now from \eqref{eq:qn-def-unbounded} the entries of $\V{R}_k$ are given by 
  \begin{align*}
    (R_k)_{\alpha,\beta} &= k^{-d/t} \int_{ k^{1/t} A} \phi_\alpha(z) \phi_\beta(z) \frac{N}{K_k(z)} v_{D,Q}\left(k^{-1/t} z\right) \dx{z} \\
                         &= k^{-d/t} \int_{ k^{1/t} A} \phi_\alpha(z) \phi_\beta(z) \frac{N v_{D,Q}\left(k^{-1/t} z\right)}{k^{-d/t} {K}^{(k)}_k\left(k^{-1/t} z\right)} \dx{z} \\
                         &= \int_{ k^{1/t} A} \phi_\alpha(z) \phi_\beta(z) w(z) \frac{N v_{D,Q}\left( k^{-1/t} z\right)}{{K}^{(k)}_k\left(k^{-1/t} z\right) w^k\left( k^{-1/t} z\right)} \dx{z} \\
                         &= \int_{D} \phi_\alpha(z) \phi_\beta(z) w(z) \left[ \frac{N v_{D,Q}\left( k^{-1/t} z\right)}{{K}^{(k)}_k\left(k^{-1/t} z\right) w^k\left( k^{-1/t} z\right)} \mathbbm{1}_{k^{1/t}A}(z) \right]\dx{z}
  \end{align*}
  Since the term in brackets weakly converges to 1 on any compact set, we have
  \begin{align*}
    (R_k)_{\alpha,\beta} \rightarrow \int_D \phi_\alpha \phi_\beta w\, \dx{z} = \delta_{\alpha,\beta}.
  \end{align*}
\end{proof}

It is clear that the results for the bounded case in Theorem \ref{thm:cls-bounded-stability} may be extended to the unbounded case.
\begin{theorem}\label{thm:unbounded-stability}
  For $D$ a conic unbounded domain with weight decomposition $w = \rho^2$ satisfying \eqref{eq:q-homogeneous}, define $\V{R}_k$ through \eqref{eq:qn-def-unbounded}. With CLS Algorithm \ref{alg:cls-unbounded} operating under the condition \eqref{eq:christoffel-sample-requirement}, then
  \begin{align*}
    \mathrm{Pr} \left[ \frac{\vertiii{\V{G} - \V{R}_k}}{\vertiii{\V{R}_k}} \geq \delta\right] \leq \frac{2}{S^{r}},
  \end{align*}
  for any $\delta \in (0,1)$ and $r > 0$.
\end{theorem}
This theorem is the unbounded analogue of Theorem \ref{thm:cls-bounded-stability}. As before, the minimum eigenvalue of $\V{R}_k$ will play a role in determining the sample count requirement through \eqref{eq:christoffel-sample-requirement}. 

A convergence result for the unbounded case that mirrors Theorem \ref{thm:bounded-convergence} may likewise be proven using the same method.
\begin{theorem}\label{thm:unbounded-convergence}
  Let $f \in L^2_w$ be given. If the number of samples $S$ in the CLS-unbounded algorithm satisfies \eqref{eq:christoffel-sample-requirement}, then the mean-square $L^2_w$ error of the truncated CLS approximation $\widetilde{f} = T_L \widetilde{\Pi}^S f$ satisfies
  \rev{
  \begin{align}
    \E \left[ \|f - \widetilde{f}\|_w^2 \right] &\leq \|f - \Pi f\|_w^2  + \frac{\varepsilon(S)}{\lambda_{\mathrm{min}}(\V{R})} \|f - \Pi f\|^2_{\widetilde{w}} + \frac{8 L^2}{S^r} + 4\kappa^2(\V{R}) d^2(f)
  \end{align}
}
with $\varepsilon(S) \triangleq \frac{2 - 2 \log 2}{(1+r) \log S} \rightarrow 0$ as $S \rightarrow \infty$, $\kappa(\V{R}) = \lambda_{\mathrm{max}}(\V{R}) / \lambda_{\mathrm{min}}(\V{R})$ the 2-norm condition number of $\V{R}$, \rev{and $d(f)$ as in \eqref{eq:data-discrepancy}}.
\end{theorem}
\rev{Just as with the bounded case, this result is influenced by $w$ versus $\widetilde{w}$ discrepancy terms; the empirical observations in Section \ref{sec:cls-accuracy} regarding the size of these additional terms holds in this case as well.}

%% file: content/results.tex
\section{Examples}\label{sec:results}
In the following section we investigate the stability and convergence properties of the CLS algorithm. The method we compare against will be a standard unweighted Monte Carlo method, Algorithm \ref{alg:unweighted-mc}. We are interested primarily in investigating how linear and log-linear sampling rates of $S$ versus the approximation space dimension $N$ affect stable and accurate reconstruction. In our figures and results, we will use ``MC" to denote an unweighted Monte Carlo procedure (i.e., as specific in Algorithm \ref{alg:unweighted-mc}), and the notation ``CLS" to denote the result of the Christoffel Least Squares algorithm (i.e., either Algorithm \ref{alg:cls-bounded} or \ref{alg:cls-unbounded}).

The sampling strategies we use are from Table \ref{tab:cls-sampling}. Note that for unbounded domains, our sampling strategies are only \textit{conjectures} because explicit formulae for weighted equilibrium measures in these cases are currently unknown.

\subsubsection*{Sampling from the ``Hermite" distribution on $\R^d$}
For the ``Hermite" case with density $w = \exp(-z^2)$, the following is one way to sample from $v_{D,Q}$ shown in Table \ref{tab:cls-sampling}: let $W \in \R^d$ be a $d$-variate standard normal random variable. The random variable $\frac{W}{\|W\|_2}$ is uniformly distributed on the surface of the unit ball $\partial B^d$. \revision{Thus, we need only find an appropriate random variable $R$ whose distribution matches the marginal distribution of $\|z\|$. Since marginalizing a spherically symmetric density on the unit ball to the radial coordinate introduces an $r^{d-1}$ factor, then the marginal density for $R$ has the form
  \begin{align*}
    \rho_R(r) &= C r^{d-1} (2 - r^2)^{d/2}, & 0 &\leq r \leq \sqrt{2},
  \end{align*}
  where $C$ is a normalization constant. However, with the change of variables $R^2 \gets P$, we see that $\frac{1}{\sqrt{2}} P$ has Beta distribution with parameters $\alpha = \frac{d}{2}$ and $\beta = \frac{d}{2}+1$. Therefore, the following prescription generates samples $Z$ according to the conjectured equilibrium measure:
  \begin{enumerate}
    \item Generate a $d$-variate standard normal random variable $W$
    \item Generate a Beta$\left(\frac{d}{2}, \frac{d}{2}+1\right)$ random variable $P$
    \item Set $Z = \sqrt{2 P} \frac{W}{\|W\|_2}$
  \end{enumerate}
} 

\subsubsection*{Sampling from the ``Laguerre" distribution on $[0, \infty)^d$}
  For the ``Laguerre" case with density $w = \exp(-\sum_j z_j)$, we need to sample from the appropriate density $v_{D,Q}$ in Table \ref{tab:cls-sampling}. \revision{However, we note that the form 
    \begin{align*}
      \dfdx{\mu_{D,Q}}{y} = C \left(4 - \sum_{j=1}^d y_j\right)^{d/2} \prod_{j=1}^d \left(y_j\right)^{-1/2} 
    \end{align*}
    is the density for a $(d+1)$-dimensional Dirichlet distribution on the variables  $\left(y_1, \ldots, y_d, 4-\|y\|_{\ell^1}\right)$ with the $d+1$ parameters $\left(\frac{1}{2}, \frac{1}{2}, \ldots, \frac{1}{2}, \frac{d}{2}+1\right)$. Therefore the following prescription generates samples $Z$ according to the conjectured equilibrium measure:
  \begin{enumerate}
    \item Generate a $(d+1)$-variate Dirichlet random variable $W$ with parameters $\left(\frac{1}{2}, \frac{1}{2}, \ldots, \frac{1}{2}, \frac{d}{2}+1\right)$.
    \item Truncate the last ($(d+1)$'th) entry of $W$
    \item Set $Z = 4 W$.
  \end{enumerate}
}

\subsection{Matrix stability}
In this section we investigate the condition number 
$
\kappa (\mathbf{\sqrt{K} V})=\frac{\sigma_{\mathrm{max}}(\mathbf{\sqrt{K} V})}{\sigma_{\mathrm{min}}(\mathbf{\sqrt{K} V})}
$ 
of the weighted design matrix $\V{V}$ from both the CLS and the unweighted MC methods, where $\sigma_{\mathrm{max}}$ and $\sigma_{\mathrm{min}}$ are the maximum and minimum singular values of a matrix, respectively. Because the design matrices for both algorithms are random matrices, we report the \textit{mean} condition number over a size-100 ensemble of tests.

\begin{figure}[t]
\begin{center}
  \includegraphics[width=\textwidth]{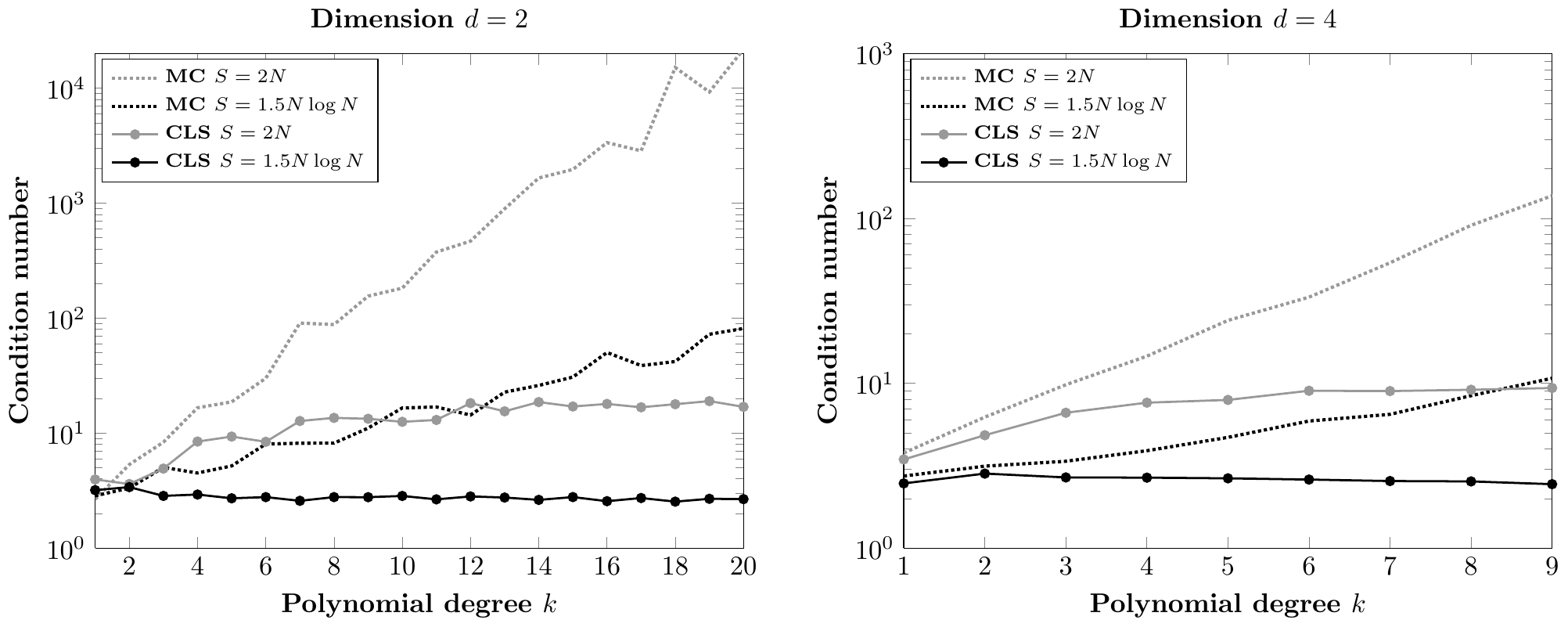}
\end{center}
\caption{Condition number with respect to the polynomial degree $k$ in the 2-dimensional (left) and 4-dimensional (right) total degree polynomial spaces (uniform measure). 
Sampling is shown for rates $S=2N$ and $S=1.5N\log \!N.$}\label{fig:TD2_Condition_Legendre}
\end{figure}
\subsubsection{Bounded domains}
We first consider the uniform distribution where Legendre polynomials are used. In Fig.~\ref{fig:TD2_Condition_Legendre} we show the condition number with respect to the polynomial degree $k$ for total degree spaces $\Lambda_k$. The left plot shows two-dimensional results while the right plot shows four-dimensional results. Both plots show results for linear scaling of sample count, i.e. $S = 2N,$ and for \textit{log-linear} dependence $S = 1.5N \log N.$ The CLS algorithm is much more stable compared to the standard MC method. Moreover, the log-linear scaling admits decay properties of the condition number with respect to the polynomial order $k$ with the CLS sampling strategy. In contrast, the linear rule admits a growth of the condition number with respect to the polynomial order $k$, for both the two kinds of design points. 
\begin{figure}[ht]
\begin{center}
\includegraphics[width=\textwidth]{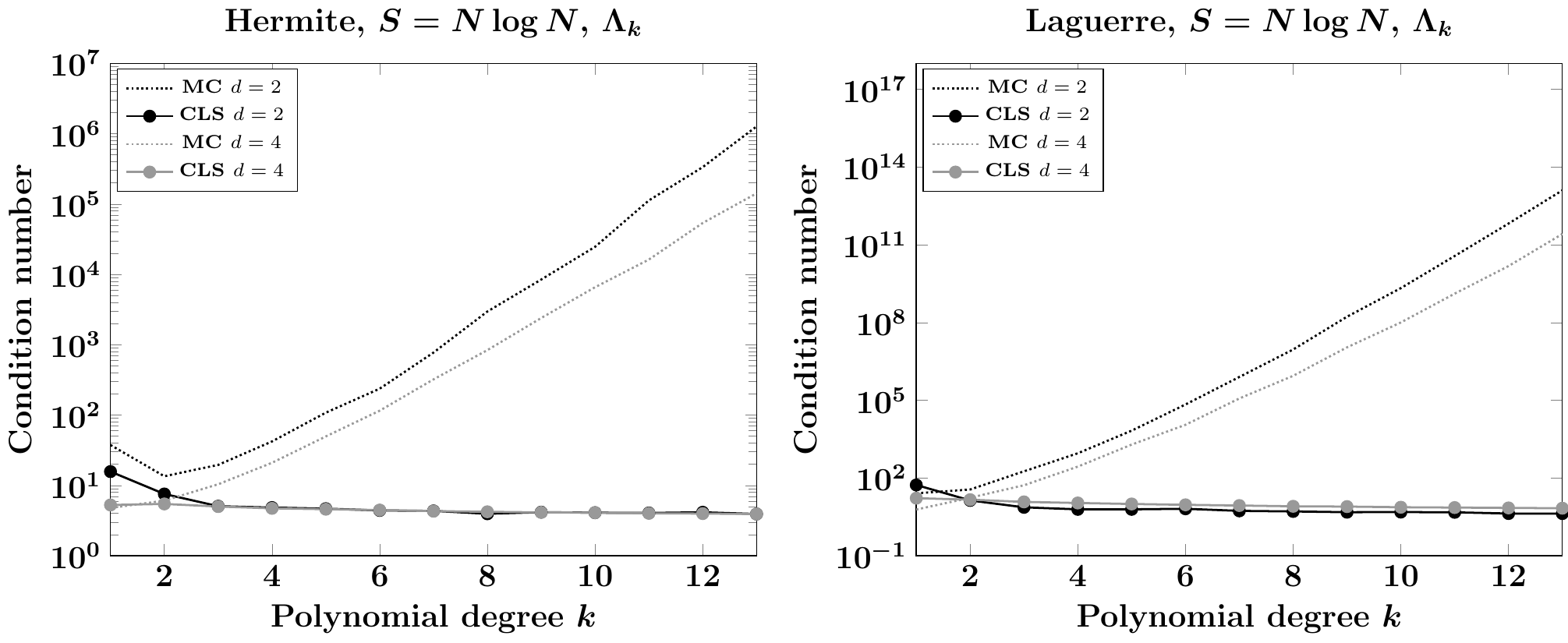}
\end{center}
\caption{\revision{Condition number against polynomial degree $n$ for the total degree polynomial space. Left: Gaussian density (Hermite polynomials). Right: Exponential density (Laguerre polynomials).}}\label{fig:Condition_Hermite}
\end{figure}

\subsubsection{Unbounded domains}
For stability on unbounded domains, we will consider the Gaussian density function $\exp(-\sum z_j^2)$ corresponding to Hermite polynomials, 
and an exponential density function $\exp(-\sum z_j)$ corresponding to Laguerre polynomials. In Fig.~\ref{fig:Condition_Hermite}, we report the condition number of the 
design matrix with respect to the polynomial degree $k$ in both the 2-dimensional total degree space and the 4-dimensional total degree space with log-linear scaling $S=N \log N$. 
The left-hand figure show the Hermite results, and the right-hand figure shows the Laguerre results. 
Again, our approach works much better, but we see that the increased dimensionality of the problem makes the CLS algorithm more ill-conditioned in the Gaussian case. 


In Fig.~\ref{fig:Condition_dimension}, we test how the dimension $d$ affect the condition number for the CLS algorithm. In the left plot, we report the numerical condition number 
for Legendre approach with $S=N\log\!N$ for $d=2,4,6.$ The dimension has little effect on the condition number, and the approach remains stable with the same dependence. 
In the right plot, we provide results for Laguerre polynomials. 
For this unbounded case, the dimension parameter $d$ appears to affect stability only weakly, just as with the bounded (Legendre) case.

\begin{figure}[ht]
\begin{center}
  \includegraphics[width=\textwidth]{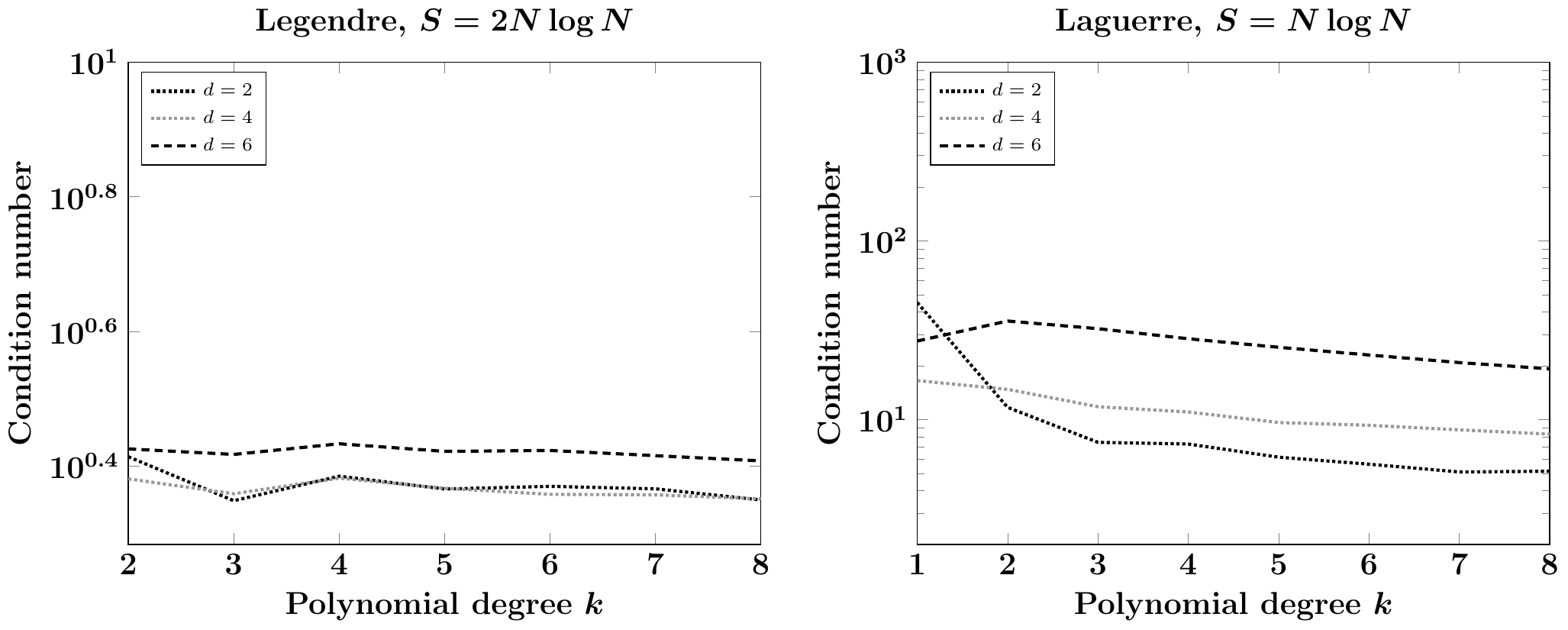}
\end{center}
\caption{\revision{CLS Condition numbers against polynomial degree $k$ for different dimensions. Left: Uniform density (Legendre polynomials), $S=2N\log\!N$. Right: Exponential density (Laguerre polynomials), $S=N\log\!N.$}}\label{fig:Condition_dimension}
\end{figure}

\subsubsection{$\ell^p$ polynomial spaces}
Up until this point we have only provided numerical examples using total-degree polynomial spaces
$\Lambda_k=\{\phi_\alpha:|\alpha|\le k \}$. In the following we will consider the effect of using polynomial spaces whose indices are defined by $\ell^p$ contour lines (with $p < 1$) on the stability of the design matrix. We define the polynomial space of strength $p>0$ as $\Lambda_k^p=\{\phi_\alpha:\lVert\alpha\rVert_p\le k \}$, where $\|\cdot\|_p$ is the discrete $\ell^p$ norm, and setting $p=1$ reverts to a total-degree space.

Figure \ref{fig:cond-vs-polynomial-space} plots the condition number of the design matrices against polynomial degree $n$ for $10$-dimensional total-degree and $\ell^p$ ($p=2/5$) Laguerre polynomial spaces. In 10 dimensions the CLS algorithm produces larger condition numbers than MC for a given total-degree space for low polynomial degree. This is in contrast to the lower dimensional results shown in Figure \ref{fig:Condition_dimension}. However CLS is again more stable than MC when we use the space $\Lambda_k^{2/5}$. The cardinality of these spaces grows much slower than the cardinality of the total-degree spaces. This slower growth allows us to provide numerical results that consider a much larger range of degrees which are computationally unfeasible using total-degree spaces. Since the benefit of the Christoffel function is asymptotic in the degree, we believe that the inclusion of terms that are high-degree in one variable and low-degree in the others (as is the case for these $\ell^p$ spaces) results in better performance of the CLS algorithm. We also note that the right-hand window of Figure \ref{fig:Condition_dimension} is the more practical case in high-dimensional approximation: using $\ell^p$-type index sets.

\begin{figure}[ht]
\begin{center}
  \includegraphics[width=\textwidth]{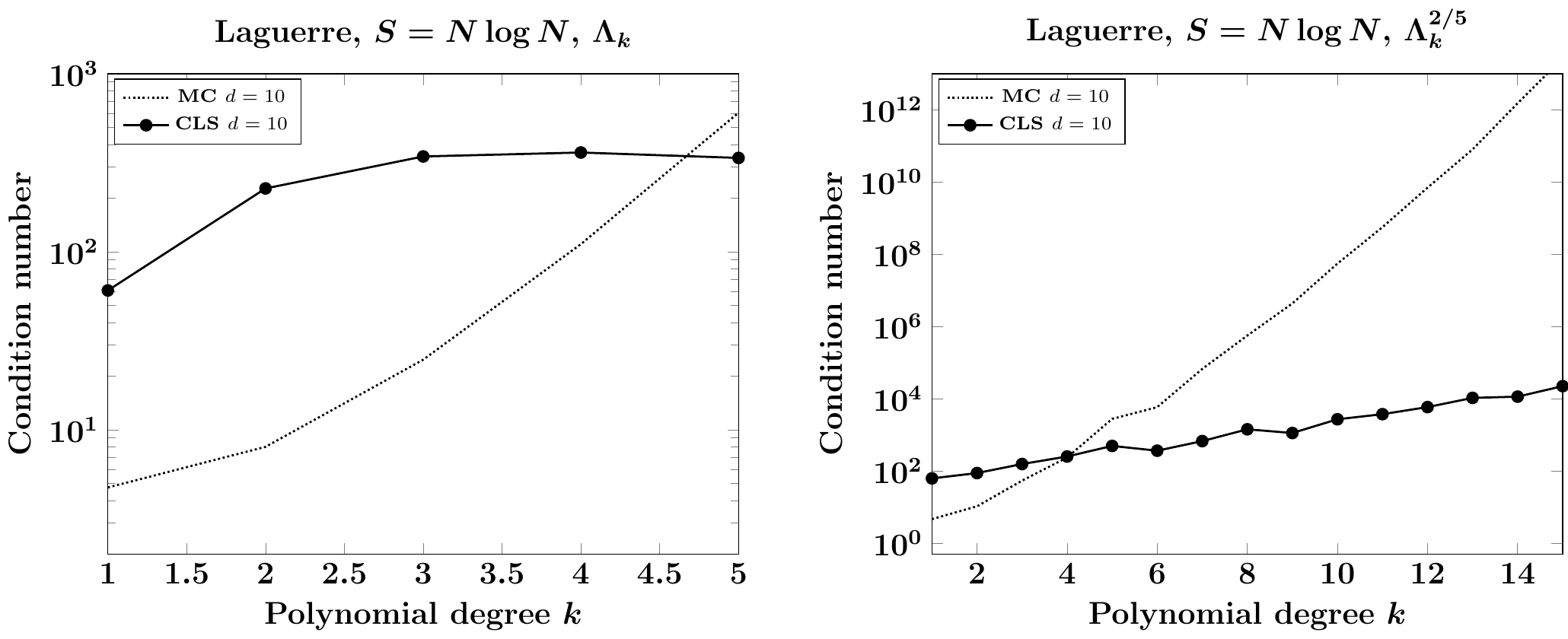}
\end{center}
\caption{\revision{Condition number against polynomial degree $k$ for total-degree $\Lambda_k$ (left) and $\ell^p$ polynomial space $\Lambda_k^{2/5}$ (right) 10-dimensional Laguerre polynomial spaces}}
\label{fig:cond-vs-polynomial-space}
\end{figure}

\subsection{Least-squares accuracy}

In this section we will compare the CLS and MC algorithms in terms of their ability to approximate a number of test functions. 
In all examples that follow we report the \textit{mean} condition number over a size-20 ensemble of tests.
\subsubsection{Algebraic function}
In Figure \ref{fig:algebraic-function-error} (left), we report the convergence rate of the least-squares projection for Legendre approximation in the 2-dimensional total degree space,
 for the test function $f(z)=\textmd{exp}\left(-\sum_{i=1}^d z_i^2\right).$ We measure accuracy using the discrete $\ell_2$ norm which is computed using $10,000$ random samples drawn from
the probability measure of orthogonality. The CLS algorithm is very stable and the error in the approximation can be driven to machine accuracy. In contrast the MC 
strategy becomes unstable as the polynomial degree is increased. Furthermore, MC sampling requires more samples to achieve a given error tolerance.
In Figure \ref{fig:algebraic-function-error} (right), we consider the Hermite approximation for the test function $f(z)=\textmd{exp}\left(-\sum_{i=1}^d z_i\right),$  
in the 3-dimensional total degree space. Again, our approach remains stable, while the MC sampling strategy becomes unstable as the polynomial degree is increased. However, in the case the CLS estimator has noticeably worse accuracy.

\begin{figure}[ht]
\begin{center}
  \includegraphics[width=\textwidth]{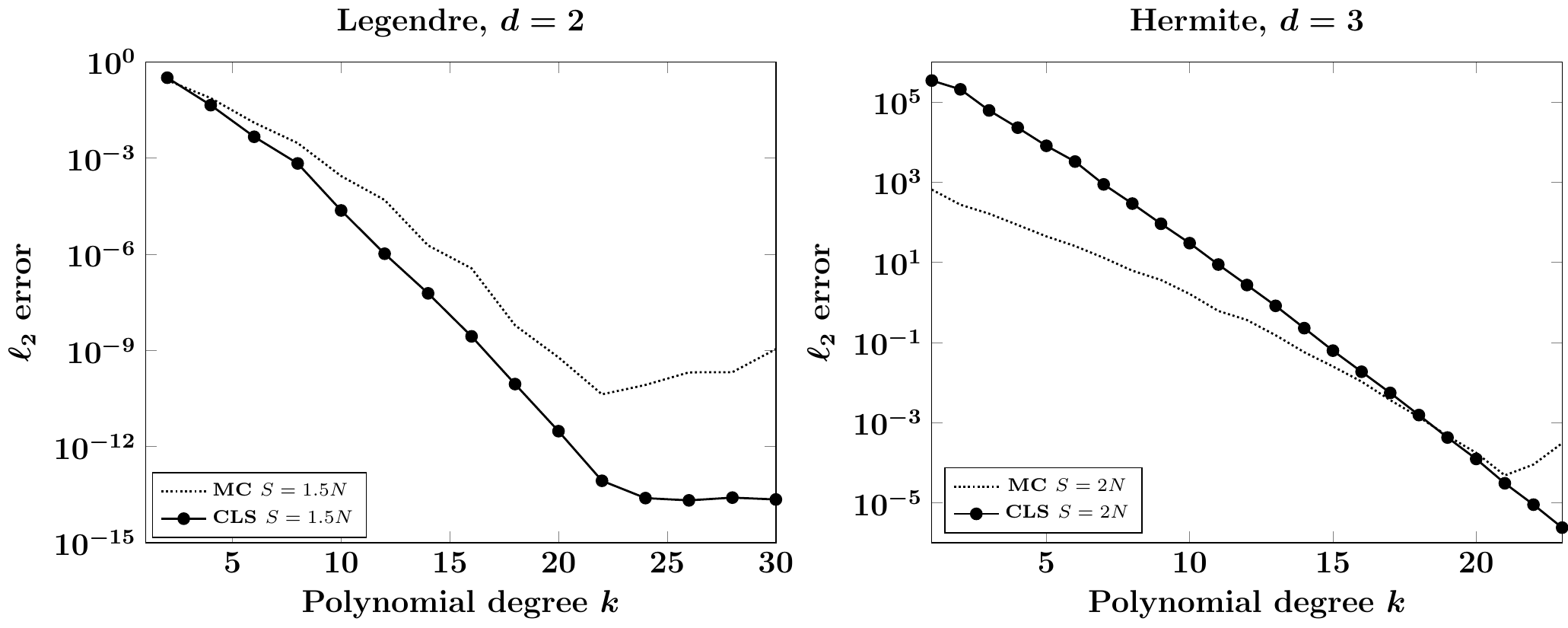}
\end{center}
\caption{\revision{Approximation error against polynomial degree $k.$ Left: Legendre approximation of $f(Y)=\textmd{exp}\left(-\sum_{i=1}^d Y_i^2\right).$ 
Right: Hermite approximation of $f(Y)=\textmd{exp}\left(-\sum_{i=1}^d Y_i\right).$}}\label{fig:algebraic-function-error}
\end{figure}

\subsubsection{Diffusion equation}
Consider the heterogeneous diffusion equation in one-spatial dimension 
\begin{equation}\label{eq:heterogeneous-diffusion}
-\frac{d}{dx}\left[a(x,\mathbf{z})\frac{du}{dx}(x,\mathbf{z})\right] = 1,\quad 
(x,\mathbf{z})\in(0,1)\times I_\mathbf{z},\quad u(0,\mathbf{z})=0,\quad u(1,\mathbf{z})=0.
\end{equation}
with an uncertain diffusivity coefficient that satisfies 
\begin{equation}\label{eq:diffusivityZ}
a(x,\mathbf{z})=\bar{a}+\sigma_a\sum_{k=1}^d\sqrt{\lambda_k}\phi_k(x)z_k,
\end{equation} 
where $\{\lambda_k\}_{k=1}^d$ and $\{\phi_k(x)\}_{k=1}^d$ are, respectively, 
the eigenvalues and eigenfunctions of the squared exponential covariance kernel 
$C_a(x_1,x_2) = \exp\left[-\frac{(x_1-x_2)^2}{l_c^2}\right].$
In the following we set $d=2$, $\bar{a}=1.0$, $\sigma=0.1$, $l_c=1$ and approximate the solution $u(1/3,\mathbf{z})$ when 
$\mathbf{z}=(z_1,z_2)$ are independent and normally-distributed random variables.
\footnote{We solve the model~\eqref{eq:heterogeneous-diffusion} using
quadratic finite elements with a high enough spatial resolution to neglect
discretization errors in our analysis.}

Figure \ref{fig:resistor-network-error} (left) compares the convergence accuracy of the Hermite polynomial least squares projection of the quantity of interest $u(1/3,\mathbf{z})$
using the CLS and MC algorithms. The accuracy of the approximation obtained using CLS is stable, whereas the MC based approximation becomes unstable as the polynomial degree is increased.
\begin{figure}[ht]
\begin{center}
\includegraphics[width=\textwidth]{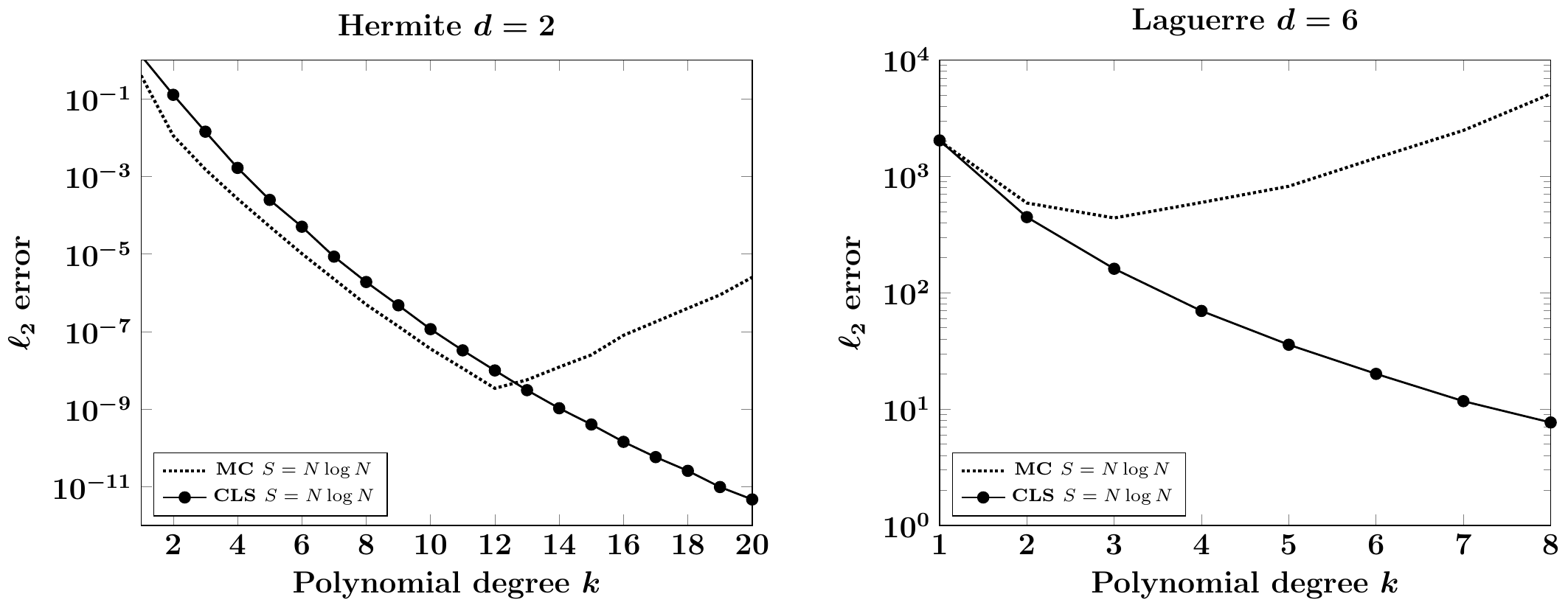}
\end{center}
\caption{\revision{Approximation error against polynomial degree $k$. Left: Hermite approximation of the $2$-dimensional diffusion equation. 
Right: Laguerre approximation of the $6$-dimensional resistor network.}}
\label{fig:resistor-network-error}
\end{figure}

\subsubsection{Resistor network}
\begin{figure}[t]
\includegraphics[width=0.95\textwidth]{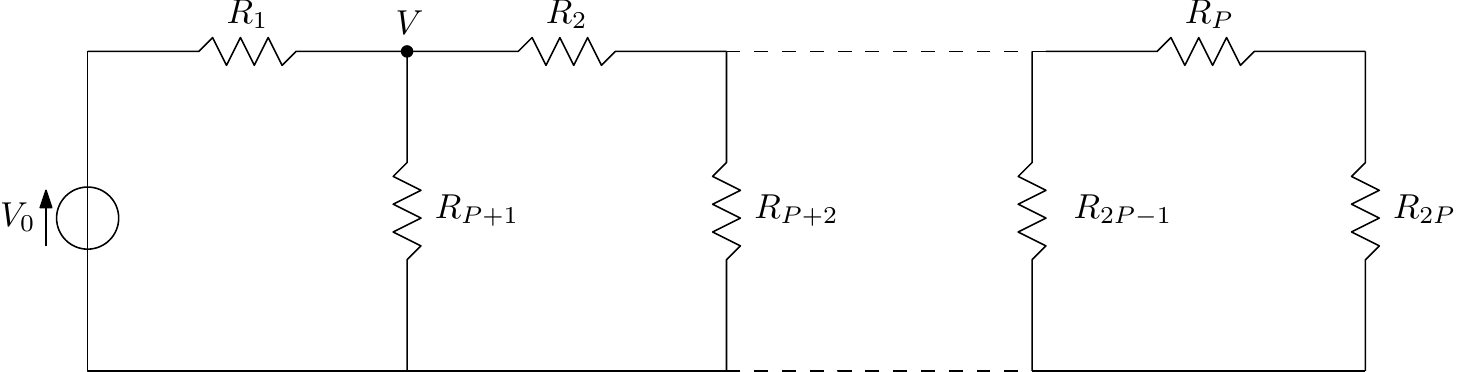}
\caption{Resistor network}\label{fig:resistor-network} 
\end{figure}

Consider the electrical resistor network shown in Fig.~\ref{fig:resistor-network}.  The network is comprised of $d=2P=6$ resistances $R_j$ of uncertain Ohmage and the network is driven by a voltage source providing a known potential $V_0=1$.  We are interested in using Laguerre polynomials to construct a least squares approximation of the voltage $V$ when the resistances are independent and identically distributed exponential random variables. As shown in all the previous examples the approximation obtained using CLS is stable for the ranges of degrees considered, whereas the MC based approximation becomes unstable as the polynomial degree is increased.

%% file: content/conclusion.tex
\section{Conclusion}\label{sec:conclusion}

Monte Carlo approximation for discrete least-squares polynomial approximation is an effective tool for approximating high-dimensional functions, and of great interest is the number of samples required for stability and convergence. We have shown that the Christoffel Least Squares algorithm can effectively approximate functions on bounded and unbounded multivariate domains, with very general multi-index sets that define the approximation space. Our theoretical results suggest that the CLS algorithm is optimal when the polynomial degree is large; our numerical results validate that the method is either superior to or competitive with standard Monte Carlo techniques in many situations of interest.

We expect it is possible to improve several of the statements about convergence using more precise estimates of Christoffel functions, which is the subject of ongoing work.